\documentclass[12pt]{amsart}

\usepackage{color,graphicx,mathrsfs,eucal,tikz,amscd,amssymb}
\usetikzlibrary{calc}
\usetikzlibrary{through}

\usepackage{graphicx}
\usepackage{amsmath}
\allowdisplaybreaks

\usepackage{color}
\usepackage{bm}
\usepackage{amsfonts,amssymb}
\usepackage{mathrsfs}

\usepackage{amsthm}	
\usepackage{cite}
\usepackage{hyperref}

\usepackage[nameinlink]{cleveref}
\usepackage{multirow}
\usepackage{makecell}

\numberwithin{equation}{section}
\numberwithin{table}{section}

\newtheorem{theorem}{Theorem}[section]
\newtheorem{lemma}[theorem]{Lemma}
\newtheorem{proposition}[theorem]{Proposition}

\newtheorem{assumption}[theorem]{Assumption}
\theoremstyle{definition}
\newtheorem{definition}[theorem]{Definition}

\theoremstyle{remark}

\usepackage{stmaryrd}

\begin{document}
	\title[Mixed method for Quad-curl]
	{Analysis of an interior penalty DG method for the quad-curl problem}
	
	\author{Gang Chen}
	\address{
		School of Mathematics, Sichuan University, Chengdu 610064, China, and School of Mathematical Sciences, University of Electronic Science and Technology of China, Sichuan 611731, China}
	\email{cglwdm@scu.edu.cn}

	\author{Weifeng Qiu}
	\address{Department of Mathematics, City University of Hong Kong, 83 Tat Chee Avenue, Hong Kong, China}
	\email{weifeqiu@cityu.edu.hk}
	
	\author{Liwei Xu}
	\address{School of Mathematical Sciences, University of Electronic Science and Technology of China, Sichuan 611731, China}
	\email{xul@uestc.edu.cn}
	
	
	
	\date{}
	
	\dedicatory{}

	\begin{abstract}
	The quad-curl term is an essential part of the resistive magnetohydrodynamic (MHD) equation and the fourth order inverse electromagnetic scattering problem, which are both of great significance in science and engineering. It is desirable to develop efficient and practical numerical methods for the quad-curl problem. In this paper, we first present some new regularity results for the quad-curl problem on Lipschitz polyhedron domains and then propose a mixed finite element method for solving the quad-curl problem. With a {\em novel} discrete Sobolev imbedding inequality for the piecewise polynomials, we obtain stability results and derive error estimates based on a relatively low regularity assumption of the exact solution.
	\end{abstract}
	

	\maketitle

\section{Introduction}
\label{intro}
Let $\Omega$ be a bounded simply connected Lipschitz polyhedron in $\mathbb{R}^3$ with connected boundary $\partial\Omega$. We consider the following quad-curl (fourth order) problem: find the vector $\bm u$ and the Lagrange multiplier $p$ such that
\begin{subequations}\label{PDE:orignial}
	\begin{align}
	\nabla\times(\nabla\times(\nabla\times(\nabla\times\bm u)))+\nabla p&=\bm f&\text{ in }\Omega,\label{o1}\\
	\nabla\cdot\bm u&=0&\text{ in }\Omega,\label{o2}\\
	\bm n\times\bm u&=\bm 0 &\text{ on }\partial\Omega,\label{o3}\\
	\bm n\times(\nabla\times\bm u)&=\bm 0&\text{ on }\partial\Omega,\label{o4}\\
	p&=0&\text{ on }\partial\Omega.\label{o5}
	\end{align}
	Here $\bm f\in [L^2(\Omega )]^3$, the vector $\bm n$ denotes the unit outer normal on $\partial\Omega$.
\end{subequations}
This model problem arises in many different applications, such as in the resistive magnetohydrodynamics (MHD) and
{\color{black} the Maxwell’s transmission eigenvalue theory.}

The resistive MHD system reads (\cite{MR2813342,MR3041683}): find the velocity $\bm u$, the pressure $p$ and the magnetic induction field $\bm B$ such that
\begin{subequations}\label{SMHD}
	\begin{align}
	\rho(\bm u_t+(\bm u\cdot\nabla)\bm u)+\nabla p&=\frac{1}{\mu_0}(\nabla\times\bm B)+\mu\Delta\bm u&\text{in }\Omega,\\
	\bm B_t-\nabla\times(\bm u\times\bm B)&=-\frac{\eta}{\mu_0}(\nabla\times(\nabla\times\bm B)) \\
	&\quad-\frac{d_i}{\mu_0}\nabla\times((\nabla\times\bm B)\times \bm B  )\nonumber\\
	&\quad-\frac{\eta_2}{\mu_0}\nabla\times(
	\nabla\times(\nabla\times(\nabla\times\bm B))
	)&\text{in }\Omega,\nonumber \\
	\nabla\cdot\bm u&=0&\text{in }\Omega,\\
	\nabla\cdot\bm B&=0&\text{in }\Omega,
	\end{align}
	with some proper boundary conditions.
\end{subequations}
Here, $\rho$ is the mass density, $\eta$ is the resistivity, $\eta_2$ is the hyper-resistivity,
$\mu_0$ is the magnetic permeability of free space, and $\mu$ is the viscosity.

In the inverse electromagnetic scattering theory, the transmission eigenvalue problem for the anisotropic
Maxwell  equations can be formulated in the following fourth order problem (\cite{MR2970278}):  find the vector $\bm u$ and the number $k$ such that
\begin{subequations}\label{inverse}
	\begin{align}
	(\nabla\times(\nabla\times)-k^2N)(N-I)^{-1}
	(\nabla\times(\nabla\times\bm u)-k^2\bm u)&=0&\text{in }\Omega,\\
	\bm n\times\bm u&=\bm 0&\text{on }\partial\Omega,\\
	\bm n\times(\nabla\times\bm u)&=\bm 0&\text{on }\partial\Omega,
	\end{align}
\end{subequations}
where $N$ is a given real matrix field and $I$ is the identity matrix. The leading term in both \eqref{SMHD} and \eqref{inverse} is $\nabla\times( \nabla\times( \nabla\times(\nabla\times\bm u))  )$.

There are vast literatures on numerical methods solving the MHD model {\em without} the quad-curl term $\nabla\times( \nabla\times( \nabla\times(\nabla\times\bm u))  )$, see \cite{MR2290574,2019-weifeng,Xu} and references therein for detailed information.
However, when the quad-curl term $\nabla\times( \nabla\times( \nabla\times(\nabla\times\bm u))  )$ is present, the design and analysis of numerical methods for the MHD model becomes more difficult and challenging.  Therefore, it is worth devising accurate and efficient numerical methods for the quad-curl problem, providing substantial tools for the solution of the resistive MHD system and the fourth order inverse electromagnetic scattering problem.

It is known that there is a strong correlation between the regularity of exact solutions and the extent of smoothness on the computational domain on which the quad-curl problem is imposed. At the continuous level of differential equations, the author proved in \cite{MR3760167} that: when the domain has no point and edge singularities, it holds that $\bm u\in [H^4(\Omega)]^3$; when the domain has a point or edge singularities, $\bm u$ does {\em not} belong to $[H^3(\Omega)]^3$ in general. In \cite{MR3808156}, the author proved that on convex polyhedral domains, if  $\nabla\cdot\bm f=0$, there hold
\begin{align*}
\bm u\in[H^2(\Omega)]^3,\qquad \nabla\times\bm u\in [H^2(\Omega)]^3,\qquad p=0.
\end{align*}
These results imply that a reasonable assumption on the regularity of the exact solution of the quad-curl problem, from which the stability and convergence results of numerics are derived,  is highly desirable for designing practical numerical methods. This is indeed one of our motivations for this work.

There are already many works devoted to the numerical study on the quad-curl problem in the past decades. In \cite{MR2813342}, a nonconforming finite element method was studied under the regularity assumption
\begin{align*}
\bm u \in [H^4(\Omega)]^3.
\end{align*}
A discontinuous Galerkin (DG) method using $\bm H(\text{curl})$ conforming elements for the quad-\text{curl} model problem was investigated in \cite{MR3041683},
where the following regularity requirements were assumed:
\begin{align}\label{post-minimal}
\bm u \in [H^2(\Omega)]^3,\qquad \nabla\times\bm u\in [H^2(\Omega)]^3.
\end{align}
{\color{black} An interior penalty DG} finite element method for the quad-\text{curl} eigenvalue problem was introduced and analyzed in \cite{MR3439219} given that the following set of regularities
\begin{align*}
\bm u \in [H^3(\Omega)]^3,\qquad\nabla\times\bm u\in [H^3(\Omega)]^3
\end{align*}
holds. Instead of solving the quad-curl problem directly, through introducing extra unknowns, a mixed finite element method was proposed and analysed in \cite[method in (44)]{MR3745016} based on a Helmholtz decomposition.
	The author proved the well-posedness and stability of the method.  Also, the convergence rate in the energy-norm is also proved on the convex domain.  A finite element method for the quad-\text{curl} problem in two dimensions was studied in \cite{MR3719596} based on the Hodge decomposition.
Concerning conforming finite element methods, since the curl-curl conforming elements in three dimensions are still unknown (see \cite{MR3949709} for curl-curl conforming elements in two dimensions),  it would be complicated and far from being obvious (since the conforming elements for the biharmonic problem are quite complicated even in two dimensions, see \cite{MR543934} for example) if the curl-curl conforming elements {\color{black} in three dimensions } are considered. {
\color{black} We refer to \cite{MR3894188} for more a mixed method for the quad-curl problem and to \cite{MR3935886,MR3864573} for the multigrid methods on the quad-curl problem.
}

\par
In this paper,  we firstly \color{black} present several regularity results of the quad-curl problem on general Lipschitz domains (might be non-convex), which have not  been documented in literature yet. Then,
we introduce {\color{black} an interior penalty DG}  finite element method solving the quad-curl  model problem \eqref{PDE:orignial}. {\color{black} Even though our numerical scheme shares some features with that proposed in \cite{MR3041683},  the authors of  \cite{MR3041683} dealt with a quad-curl problem with a reaction term which makes their theoretical analysis different from ours.}  Finally, we prove the corresponding stability and convergence results of the numerical solution of $\bm u$ under a relatively low regularity compared to that in existing works, i.e.
\begin{align*}
\bm u\in[H^{r_{u_0}}(\Omega)]^3,\quad
\nabla\times\bm u\in [H^{r_{u_1}}(\Omega)]^3,\quad
\nabla\times(\nabla\times\bm u)\in[H^{r_{u_2}}(\Omega)]^3,\quad
p\in H^{ r_p}(\Omega),
\end{align*}
where $r_{u_0}>\frac{1}{2}$, $r_{u_1}\ge 1$, $r_{u_2}>\frac{1}{2}$, and $r_p>\frac{3}{2}$.    In addition, we establish a {\em novel} discrete Sobolev imbedding inequality in the following piecewise $H^1$ norm (see \Cref{discrete-H1}):
\begin{align*}
\sum_{K\in\mathcal{T}_h}\|\bm v_h\|_{1,K}^2\le
C\left[
\sum_{K\in\mathcal{T}_h}\left(
\|\nabla\times\bm v_h\|^2_{0,K}
+\|\nabla\cdot\bm v_h\|^2_{0,K}
\right)
+\sum_{F\in\mathcal{E}_h}h_F^{-1}\|[\![\bm v_h]\!]\|^2_{0,F}
\right],
\end{align*}
where $\bm v_h$ is a piecewise polynomial of a fixed order. This inequality provides us with a fundamental tool to further achieve the discrete $H^1$ stability and $H^1$ error estimate for the approximation of $\nabla\times\bm u$. Moreover, turning to the inequality, we could apply our mixed method to solve eigenvalue problems and carry out the numerical analysis of nonlinear problems and high order problems in a low regularity region via a posteriori analysis techniques (\cite{MR2684360,MR2846773}), and we will consider it in the future works.
\par
The rest of this paper is organized as follows.
We present some regularity results for the partial differential equations (PDEs) with the quad-curl term in \Cref{sec:reg}.
A mixed method for the quad-curl problem is introduced in \Cref{sec:mixed}.
We obtain a novel discrete Sobolev inequality
and stability results for the underlying mixed method in \Cref{sec:stability}.
The convergence result is proved through an energy argument in \Cref{sec:error}.
We give estimates in $\bm H({\rm curl})$ norm through dual arguments in \Cref{sec:dual}.

\section{Regularity for PDEs\label{sec:reg}}

For any bounded domain $\Lambda \subset  \mathbb{R}^s$ $(s=2,3)$, and any two functions $u,v\in L^2(\Lambda)$, we denote the $L^2(\Lambda)$ inner product and its norm by
\begin{align*}
(u,v)_{\Lambda}:=\int_{\Lambda}uv\;{ \rm d}x,\qquad
\|u\|_{0,\Lambda}:=(u,u)_{\Lambda}^{\frac{1}{2}},
\end{align*}
and when $\Lambda=\Omega$, we set
\begin{align*}
(u,v):=(u,v)_{\Omega},\qquad
\|u\|_{0}:=\|u\|_{0,\Omega},
\end{align*}
for simplicity.
We denote the Sobolev spaces defined on $\Lambda$ by $W^{m,p}(\Lambda)$ and $W_0^{m,p}(\Lambda)$, and
denote  its semi-norm and norm by $|v|_{m,p,\Lambda}$, $\|v\|_{m,p,\Lambda}$, respectively. When $p=2$ we omit $p$ in $|v|_{m,p,\Lambda}$ and $\|v\|_{m,p,\Lambda}$; when $\Lambda=\Omega$, we omit $\Lambda$ in $|v|_{m,p,\Lambda}$ and $\|v\|_{m,p,\Lambda}$.
For conventional notations, we denote
\begin{align*}
&H^m(\Lambda):= W^{m,2}(\Lambda),\qquad
H^m_0(\Lambda):= W_0^{m,2}(\Lambda).
\end{align*}
In particular, when $\Lambda\in \mathbb{R}^{2}$, we use $\langle\cdot,\cdot\rangle_{\Lambda}$ to replace $(\cdot,\cdot)_{\Lambda}$ for distinction.
The bold face fonts will be used for vector (or tensor) analogues of the Sobolev spaces along with vector-valued (or tensor-valued) functions.
We define the following function spaces \color{black}
\begin{align*}
\bm{H}(\text{div} ;\Omega)&:=\{\bm{v}\in [L^2(\Omega)]^3: \nabla\cdot\bm{v}\in L^2(\Omega) \},\\
\bm{H}(\text{curl};\Omega)&:=\{\bm{v}\in [L^2(\Omega)]^3: \nabla\times\bm{v}\in [L^2(\Omega)]^3 \},\\
\bm{H}^s(\text{curl};\Omega)&:=\{\bm{v}\in [H^s(\Omega)]^3: \nabla\times\bm{v}\in [H^s(\Omega)]^3 \text{ with }s\ge 0\},\\
\bm H({\rm curl}^2;\Omega)
&:=\{
\bm v\in [L^2(\Omega)]^3:\nabla\times\bm v\in \bm{H}(\text{curl};\Omega)
\}
\end{align*}
equipped  with the graph norms
\begin{align*}
&\|\bm v\|_{\bm H({\rm div};\Omega)}:=\left(
\|\bm v\|^2_0+\|\nabla\cdot\bm v\|^2_0\right)^{\frac{1}{2}},\qquad\quad
\|\bm v\|_{\bm H({\rm curl};\Omega)}:=\left(
\|\bm v\|^2_0+\|\nabla\times\bm v\|^2_0\right)^{\frac{1}{2}},\\
&\|\bm v\|_{\bm H^s({\rm curl};\Omega)}:=\left(
\|\bm v\|^2_s+\|\nabla\times\bm v\|^2_s\right)^{\frac{1}{2}},\qquad
\|\bm v\|_{\bm H({\rm curl}^2;\Omega)}:=\left(
\|\bm v\|^2_0+\|\nabla\times\bm v\|_{\bm H({\rm curl};\Omega)}^2
\right)^{\frac{1}{2}},
\end{align*}
respectively. Furthermore, we define
\begin{align*}
\bm H_0({\rm div};\Omega)&:=\{
\bm v\in\bm H({\rm curl};\Omega):\bm n\cdot\bm v|_{\partial\Omega}= 0
\},\\
\bm H_0({\rm curl};\Omega)&:=\{
\bm v\in\bm H({\rm curl};\Omega):\bm n\times\bm v|_{\partial\Omega}=\bm 0
\},\\
\bm H_0^s({\rm curl};\Omega)&:=\{
\bm v\in\bm H^s({\rm curl};\Omega):\bm n\times\bm v|_{\partial\Omega}=\bm 0
\},\\
\bm H_0({\rm curl}^2;\Omega)&:=\{
\bm v\in\bm H({\rm curl}^2;\Omega):\bm n\times\bm v|_{\partial\Omega}
=\bm n\times(\nabla\times\bm v)|_{\partial\Omega}
=\bm 0
\}.
\end{align*}

 Throughout this paper, we use $C$ to denote a positive constant independent of mesh size, not necessarily the same at its each occurrence.
The following imbedding theory is standard but useful in the analysis of $\bm H(\text{curl})$ space.
\begin{lemma} [{\color{black} \cite[Proposition 3.7]{MR1626990}}]\label{embed}
	If the domain $\Omega$ is a Lipschitz polyhedron, then $\bm{X}_T(\Omega)$ and $\bm{X}_N(\Omega)$ are continuously imbedded in $[H^{\alpha}(\Omega)]^3$ for a real number $\alpha\in (\frac{1}{2},1]$, where the spaces $ \bm X_N(\Omega)$ and $ \bm X_T(\Omega)$ are defined as
	\begin{align*}
	\bm{X} _N:=\bm{H}_0({\rm curl};\Omega)\cap \bm{H}({\rm div} ;\Omega),\qquad
	\bm{X} _T:=\bm{H}({\rm curl};\Omega)\cap \bm{H}_0({\rm div} ;\Omega).
	\end{align*}
	
\end{lemma}

%
\color{black}
\begin{definition}
	We define the weak formulation  of \eqref{PDE:orignial} as:
	\par
	Find $(\bm u,p)\in \bm H_0({\rm curl}^2;\Omega)\times H_0^1(\Omega)$ such that
	\begin{subequations}\label{var}
		\begin{align}
		(\nabla\times(\nabla\times\bm u),\nabla\times(\nabla\times\bm v))+(\nabla p,\bm v)&=(\bm f,\bm v),\label{var1}\\
		(\bm u,\nabla q)&=0\label{var2}
		\end{align}
		for all $(\bm v,q)\in \bm H_0({\rm curl}^2;\Omega)\times H_0^1(\Omega)$.
	\end{subequations}
\end{definition}

%
%

\begin{theorem}\label{dense} The space $[\mathcal C_0^{\infty}(\Omega)]^3$ is dense in the space $\bm H_0({\rm curl}^2;\Omega)$.
\end{theorem}
\begin{proof}  For any $\bm v\in  \bm H_0({\rm curl}^2;\Omega)$, we define
	$$ \tilde{\bm v}=\left\{
	\begin{aligned}
	\bm v(\bm x) & \text{ if }\bm x\in \Omega,\\
	\bm 0\quad        &  \text{ if }\bm x\in \mathbb R^3/\Omega.
	\end{aligned}
	\right.
	$$
	Since $[\mathcal C_0^{\infty}(\Omega)]^3$ is dense in $\bm H_0({\rm curl};\Omega)$, we have $\tilde{\bm v}\in \bm H({\rm curl};\mathbb R^3)$, thus $\nabla\times\tilde{\bm v}\in \bm L^2(\mathbb R^3)$ and
	$$ \nabla\times\tilde{\bm v}=\left\{
	\begin{aligned}
	\nabla\times\bm v(\bm x) &  \text{ if }\bm x\in \Omega,\\
	\bm 0\quad        &  \text{ if }\bm x\in \mathbb R^3/\Omega.
	\end{aligned}
	\right.
	$$
	Again, due to the fact that $[\mathcal C_0^{\infty}(\Omega)]^3$ is dense in $\bm H_0({\rm curl};\Omega)$, we have $\nabla\times\widetilde{\bm v}\in \bm H({\rm curl};\mathbb R^3)$. Then we obtain that the space $[\mathcal C_0^{\infty}(\Omega)]^3$ is dense in space $\bm H_0({\rm curl}^2;\Omega)$ by following the proof of the part (ii) of \cite[Theorem 3.29]{MR1742312}.
\end{proof}

Looking at the test function as $\mathcal C_{0}^{\infty}(\Omega)$ in \eqref{var}, we know that the solution $(\bm u,p)\in \bm H_0({\rm curl}^2;\Omega)\times H_0^1(\Omega)$ of \eqref{var} satisfies \eqref{PDE:orignial} in the distribution sense.
On the other hand, the solution $(\bm u,p)$ of \eqref{PDE:orignial}
with the low regularity
$\bm u\in [L^2(\Omega)]^3$,
$\nabla\times(\nabla\times\bm u)\in [L^2(\Omega)]^3$ and $p\in H_0^1(\Omega)$
 satisfies \eqref{var}. Now we are ready to prove the following regularity on  Lipschitz polyhedron domains for the weak solution $(\bm u,p)\in \bm H_0({\rm curl}^2;\Omega)\times H_0^1(\Omega)$.

\begin{theorem} \label{th:reg}
	Let $\Omega$ be a simply connected Lipschitz domain in $\mathbb R^3$, for any given $\bm f\in [L^2(\Omega)]^3$, the problem
	\eqref{PDE:orignial} has a unique \color{black} weak solution $(\bm u,p)\in{\bm H}_0({\rm curl}^2;\Omega)\times H_0^1(\Omega)$. In addition, the following regularity estimate holds true:
	\begin{align*}
	\|\bm u\|_{r_{u_0}}	
		+\|\nabla\times\bm u\|_1
			+\|\nabla\times(\nabla\times\bm u)\|_0
	+\|\nabla\times(\nabla\times(\nabla\times(\nabla\times\bm u)))\|_{0}	
	+\|\nabla p\|_0
	\le C\|\bm f\|_0,
	\end{align*}
	where the regularity index
	$r_{u_0}\in (\frac{1}{2},1]$.
	Furthermore, if $\nabla\times(\nabla\times\bm u)\in  [H^{\color{black}r_{u_1}-1}(\Omega)]^3$, we have the following regularity
	\begin{align*}
	\|            \nabla\times\bm u\|_{r_{u_1}}
	\le C\|\nabla\times(\nabla\times\bm u)\|_{r_{u_1}-1},
	\end{align*}
	where $r_{u_1}\in [1,\frac{3}{2})$, and it is close to $\frac{3}{2}$.
\end{theorem}
\begin{proof}\color{black}
	{ \color{black}
		The following  inf-sup condition is followed by \Cref{dense}:
		\begin{align*}
		\sup_{\bm 0\neq \bm v\in \bm H_0({\rm curl}^2;\Omega)}\frac{(\nabla p,\bm v)}{\|\bm v\|_{\bm H({\rm curl}^2;\Omega)}}\ge C\|\nabla p\|_0.
		\end{align*}
In addition, we have the following V-elliptic property
		\begin{align*}
		\|\bm u\|_0\le C\|\nabla\times\bm u\|_{0}\le C\|\nabla\times(\nabla\times\bm u)\|_0,
		\end{align*}
	for $\bm u\in \bm H_0({\rm curl}^2;\Omega)$, and $(\bm u,\nabla q)=0$ for all $q\in H_0^1(\Omega)$. The  existence of a unique solution in $\bm H_0({\rm curl}^2;\Omega)\times H_0^1(\Omega)$
	of \eqref{var} is followed immediately.
}
	First, we present our proof of the regularity estimate in following several steps:
	\begin{itemize}

		\item Proof of $\|\nabla\times(\nabla\times\bm u)\|_0\le C\|\bm f\|_0$: Taking $\bm v=\bm u$ in \eqref{var1} and $q=-p$ in \eqref{var2} and adding them together, we have
		\begin{align*}
		\|\nabla\times(\nabla\times\bm u)\|^2_0=(\bm f,\bm u).
		\end{align*}
	    Considering
		the facts that $(\bm u,\nabla q)=0$ for all $q\in H_0^1(\Omega)$,  $(\nabla\times\bm u,\nabla q)=0$ for all $q\in H_0^1(\Omega)$, and \Cref{embed} or \cite[Corollary 4.8]{MR2059447},
		we obtain
		\begin{align*}
		\|\nabla\times(\nabla\times\bm u)\|^2_0
		\le \|\bm f\|_0\|\bm u\|_0\le C \|\bm f\|_0 \|\nabla\times \bm u\|_0\le C \|\bm f\|_0\|\nabla\times(\nabla\times\bm u)\|_0,
		\end{align*}
		which leads to
		\begin{align*}
		\|\nabla\times(\nabla\times\bm u)\|_0\le C\|\bm f\|_0.
		\end{align*}

		\item Proof of $	\|\bm u\|_{r_{u_0}}\le C\|f\|_0$: It can be obtained directly from \Cref{embed}, and the fact that $	\|\nabla\times(\nabla\times\bm u)\|_0\le C\|\bm f\|_0$.

		\item Proof of  $\|\nabla\times\bm u\|_{1}
		\le C\|\bm f\|_0$: We define $\bm \sigma:=\nabla\times\bm u$. From $\bm n\times\bm u|_{\partial\Omega}=\bm 0$ we have,
		\begin{align*}
		\bm n\cdot\bm \sigma|_{\partial\Omega}=\bm n\cdot(\nabla\times\bm u)|_{\partial\Omega}=\nabla_{\partial\Omega}\cdot(\bm n\times\bm u)=0,
		\end{align*}
		which, together with $\bm n\times\bm{\sigma}|_{\partial\Omega}=\bm 0$ in terms of \eqref{o4}, leads to
		\begin{align}\label{ss1}
		\bm \sigma=\bm 0 \text{ on }\partial\Omega.
		\end{align}
		Meanwhile, $\nabla\cdot\bm\sigma=0$, gives
		\begin{align}\label{ss2}
		\Delta \bm{\sigma}=\nabla(\nabla\cdot\bm{\sigma} )-\nabla\times(\nabla\times\bm {\sigma})=-\nabla\times(\nabla\times\bm {\sigma})\text{ in }\Omega.
		\end{align}
		The combination of \eqref{ss1}, \eqref{ss2}, and the regularity for the elliptic problem in \cite{MR961439} yields
		\begin{align*}
		\|\nabla\times\bm u\|_{1}=\|\bm\sigma\|_{1}&\le C\|\nabla\times(\nabla\times\bm \sigma)\|_{-1}\le C\|\nabla\times(\nabla\times\bm u)\|_0\le C\|\bm f\|_0.
		\end{align*}

		\item Proof of  $\|\nabla\times\bm u\|_{r_{u_1}}
		\le C\|\nabla\times(\nabla\times(\nabla\times\bm u))\|_{r_{u_1}-2}$: The combination of \eqref{ss1}, \eqref{ss2} and  the regularity for the elliptic problem in \cite[Theorem 1.1]{MR1331981},
		for some constant $r_{u_1}\in [1,\frac{3}{2})$ and it is colse to $\frac{3}{2}$, yields
		\begin{align*}
		\|\nabla\times\bm u\|_{r_{u_1}}=\|\bm\sigma\|_{r_{u_1}}&\le C\|\nabla\times(\nabla\times\bm \sigma)\|_{r_{u_1}-2}\nonumber\\
		&=C\|\nabla\times(\nabla\times(\nabla\times\bm u))\|_{r_{u_1}-2}\nonumber\\
		&\le C\|\nabla\times(\nabla\times\bm u)\|_{r_{u_1}-1}.
		\end{align*}

		\item Proof of $\|\nabla p\|_0\le \|\bm f\|_0$: Letting  $\bm v=\nabla p$, it holds that $\nabla\times\bm v=\bm 0$ and $\bm v\in \bm H_0({\rm curl}^2;\Omega)$. We take $\bm v=\nabla p$ in \eqref{var1} to get
		\begin{align*}
		\|\nabla p\|^2_0=(\bm f,\nabla p) \le \|\bm f\|_0\|\nabla p\|_0.
		\end{align*}
		Then $\|\nabla p\|_0\le \|\bm f\|_0$ follows immediately.
		
		\item Proof of $	\|\nabla\times(\nabla\times(\nabla\times(\nabla\times\bm u)))\|_{0}\le C\|\bm f\|_0$: It follows from \eqref{var1} that
				$\nabla\times(\nabla\times(\nabla\times(\nabla\times\bm u)))=\bm f-\nabla p\in [L^2(\Omega)]^3$ in the distribution sense, therefore,
				\begin{align*}
				\|\nabla\times(\nabla\times(\nabla\times(\nabla\times\bm u)))\|_{0}
				=	\|\bm f-\nabla p\|_{0}\le 	\|\bm f\|_0+\|\nabla p\|_{0}\le 2\|\bm f\|_0.
				\end{align*}

	\end{itemize}
	
Finally, the uniqueness of the solution is  followed from the regularity estimate, and thus the existence of the solution follows immediately.

\end{proof}
\color{black}

\section{{\color{black} an interior penalty DG}  finite element method\label{sec:mixed}}

Let $\mathcal{T}_h$  be a quasi-uniform partition of the domain $\Omega$ consisting of tetrahedrons. For any element $K\in\mathcal{T}_h$, let $h_K$ be the infimum of the diameters of spheres containing $K$ and denote the mesh size $h:=\max_{K\in\mathcal{T}_h}h_K$. Let $\mathcal{E}_h$ be the set of all faces of the mesh $\mathcal{T}_h$
.
For any element $K\in\mathcal{T}_h$ and face $F\in\mathcal{E}_h$, we denote by $ \bm{n}_K $ and $\bm{n}_F$  the unit outward  normal vector to $\partial K$ and  face $F$, respectively.
Let $F=\partial K \cap \partial K{'}$ be an interior face shared by element $K$ and element
$K{'}$, and $\bm n_F$ be pointing from $K$  to $K{'}$. Let $\bm{\phi}$ be a piecewise smooth function. We
define the average and jump of $\bm{\phi}$ on $F$ as
\begin{align*}\color{black}
\{\!\!\{ \bm{\phi}\}\!\!\}:=\frac{1}{2}(\bm{\phi}|_K+\bm{\phi}|_{K'}),\qquad
[\![ \bm{\phi}]\!]:=\bm{\phi}|_K-\bm{\phi}|_{K'}.
\end{align*}
On a boundary face $F\subset\partial K\cap \partial\Omega$, we set $\{\!\!\{\bm{\phi}\}\!\!\}:=\bm\phi$, $[\![ \bm{\phi}]\!]:=\bm{\phi}$ and $\bm n_F=\bm n_K|_F$.
We denote by $\mathcal P_{\ell}(\Lambda)$ the set of all polynomials with order at most $\ell$ on bounded domains $\Lambda$, and denote by $\mathcal P_{\ell}(\mathcal T_h)$ the set of all piecewise polynomials with order at most $\ell$ with respect to the decomposition $\mathcal{T}_h$.
\subsection{{\color{black} Interior penalty DG}  methods}
For any integer $ k\ge 1$, we define the finite element spaces
\begin{align*}
\bm E_h:=\bm H_0(\text{curl};\Omega)\cap [\mathcal P_{k+1}(\mathcal{T}_h)]^3,\qquad Q_h:=H^1_0(\Omega)\cap \mathcal P_{k+2}(\mathcal{T}_h).
\end{align*}
Then our {\color{black} interior penalty DG}  method reads:
\par
Find $\bm u_h\in \bm E_h$ and $p_h\in Q_h$ such that
\begin{subequations}\label{FEM}
	\begin{align}
	\sum_{K\in\mathcal{T}_h}(\nabla\times(\nabla\times\bm u_h),\nabla\times(\nabla\times\bm v_h))_{K}+(\nabla p_h,\bm v_h)&\nonumber\\
	-\sum_{F\in\mathcal{E}_h}\langle\{\!\!\{ \nabla\times(\nabla\times\bm u_h)\}\!\!\},\bm n\times[\![\nabla\times\bm v_h]\!] \rangle_F&\nonumber\\
	\mp\sum_{F\in\mathcal{E}_h}\langle\{\!\!\{ \nabla\times(\nabla\times\bm v_h)\}\!\!\},\bm n\times[\![\nabla\times\bm u_h]\!] \rangle_F&\nonumber\\
	+\sum_{F\in\mathcal{E}_h}\color{black}\tau h_F^{-1}\color{black}\langle\bm n\times[\![\nabla\times\bm u_h]\!],\bm n\times[\![\nabla\times\bm v_h]\!] \rangle_F
	&=(\bm f,\bm v_h),\label{fem01}\\
	(\bm u_h,\nabla q_h)&=0\label{fem02}
	\end{align}
	hold for all $(\bm v_h,q_h)\in \bm E_h\times Q_h$.
	The stabilization parameter $\tau>0$ is independent of the mesh size.
\end{subequations}
In the following text, we consider the analysis only for the symmetry case (i.e., we replace `$\mp$' by `$-$' in \eqref{fem01}), since the proof of the non-symmetry case is similar to the symmetry case.

\subsection{Interpolations}
For integer $\ell\ge1$, we denote by $\bm{\Pi}_{h,\ell}^{\text{curl}}$ the standard  $\bm H({\rm curl})$-conforming interpolation of the second kind from
$\bm H^s(\text{curl};\Omega)$ to $\bm H(\text{curl};\Omega)\cap [\mathcal P_{\ell}(\mathcal{T}_h)]^3$ with $s>\frac{1}{2}$, and thus also from $\bm H^s({\rm curl};$ $\Omega)\cap\bm H_0(\text{curl};\Omega)$ to $\bm H_0(\text{curl};\Omega)\cap [\mathcal P_{\ell}(\mathcal{T}_h)]^3$ with $s>\frac{1}{2}$. The following approximation properties hold (\cite{MR864305,MR1609607,MR2059447})
\begin{subequations}
	\begin{align}\label{app-u1}
	\|\bm u-\bm{\Pi}_{h,\ell}^{\text{curl}}\bm u\|_{0,K}&\le Ch_K^t(\|\bm u\|_{t,K}+\color{black}h_K\color{black}\|\nabla\times\bm u\|_{t,K}),\\
	\|\nabla\times(\bm u-\bm{\Pi}_{h,\ell}^{\text{curl}}\bm u)\|_{0,K}&\le Ch_K^t\|\nabla\times\bm u\|_{t,K},\label{app-u2}
	\end{align}
	where  $\bm u\in \bm H^t({\rm curl};\Omega)$, \color{black} and real number \color{black} $t\in (\frac{1}{2},\ell]$, and
		\begin{align}\label{app-u4}
	\|\bm u-\bm{\Pi}_{h,\ell}^{\text{curl}}\bm u\|_{0,K}&\le Ch_K^m\|\bm u\|_{m,K},
	\end{align}
\end{subequations}
	where  $\bm u\in \bm H^m(\Omega)$, \color{black} and real number \color{black} $m\in (1,\ell+1]$. We define the $L^2$-projection from
$L^2(\Omega)$ onto $Q_h$ as: for any
$p\in L^2(\Omega)$, find $\Pi_h^Qp\in Q_h$ such that
\begin{align*}
(\Pi_h^Qp,q)=(p,q)\qquad\forall q\in Q_h.
\end{align*}
The following approximation result holds
\begin{align}\label{app-p}
|p-\Pi_h^Qp|_1\le Ch^{j-1}\|p\|_j
\end{align}
for \color{black} real number $j\in [1,k+3]$ and $p\in H^j(\Omega)$. Next, we introduce an interpolation (\cite{MR3771897,2019-weifeng}):
for any $\bm v\in \bm H^s({\rm curl};\Omega)$ with $s>\frac{1}{2}$, we define $\bm\Pi_h^{\bm E}\bm v\in \bm E_h$ such that
\begin{align}\label{def:piE}
\bm{\Pi}_h^{\bm E}\bm v=\bm{\Pi}_{h,k+1}^{\text{curl}}\bm v+\nabla \sigma_h,
\end{align}
where $\sigma_h\in Q_h$ is the solution of the well-posed elliptic problem:
\begin{align}
(\nabla\sigma_h,\nabla q_h)=(\bm v-\bm{\Pi}_{h,k+1}^{\text{curl}}\bm v,\nabla q_h) \qquad \forall q_h\in Q_h.
\end{align}

Utilizing \eqref{app-p} and \eqref{def:piE}, we get the following result immediately.
\begin{lemma} We have the following orthogonality
	\begin{align}\label{or-p}
	(\bm v-\bm{\Pi}_h^{\bm E}\bm v,\nabla q_h)=0
	\end{align}
	hold for all $\bm v\in\bm H^s({\rm curl};\Omega)$ \color{black} with \color{black} $s>\frac{1}{2}$ and $q_h\in Q_h$. In addition, there holds the approximation property
	\begin{align}
	\|\bm v-\bm{\Pi}_h^{\bm E}\bm v\|_0\le 2\|\bm v-\bm{\Pi}_{h,k+1}^{\rm curl}\bm v\|_0.
	\end{align}
\end{lemma}

\color{black}

%
%
%
%
%

\section{Existence of a unique solution and stability for the mixed method\label{sec:stability}}
In this section, we will derive stability results for the underlying mixed method. To this purpose, we firstly develop a novel discrete Sobolev imbedding inequality which is also an efficient tool for numerical analysis of  nonlinear problems, and we will report it in future works.
All  results  presented in \Cref{s41,s43,s44} are valid as $\Omega\in\mathbb{R}^3$ is a bounded multi-connected Lipschitz polyhedron  even though we only consider in our presentation $\Omega\in\mathbb{R}^3$ to be a bounded simply connected Lipschitz polyhedron.
{\color{black}
The $L^3$ and $L^6$ stability provided in \Cref{s43,s44} allows a better bound of $\bm u_h$ than the $L^2$ stability does.
}

\subsection{ A novel discrete Sobolev imbedding inequality\label{s41}}
We present a  novel discrete Sobolev inequality in the following subsection, which is the {\em first} to be reported in literatures.

\begin{theorem}\label{discrete-H1} Assume that $\Omega$ is a bounded Lipschitz polyhedron {\color{black} (not necessarily simply-connected)} in $\mathbb{R}^3$, then there exists  a  positive constant $C>0$ such that
	\begin{align*}
	\sum_{K\in\mathcal{T}_h}\|\bm v_h\|_{1,K}^2&\le
	C\left[
	\sum_{K\in\mathcal{T}_h}\left(
	\|\nabla\times\bm v_h\|^2_{0,K}
	+\|\nabla\cdot\bm v_h\|^2_{0,K}
	\right)
	+\sum_{F\in\mathcal{E}_h}h_F^{-1}\|[\![\bm v_h]\!]\|^2_{0,F}
	\right],\\
	\|\bm v_h\|_{0,6}^2&\le
	C\left[
	\sum_{K\in\mathcal{T}_h}\left(
	\|\nabla\times\bm v_h\|^2_{0,K}
	+\|\nabla\cdot\bm v_h\|^2_{0,K}
	\right)
	+\sum_{F\in\mathcal{E}_h}h_F^{-1}\|[\![\bm v_h]\!]\|^2_{0,F}
	\right]
	\end{align*}
	hold for all $\bm v_h\in [\mathcal{P}_{\ell}(\mathcal{T}_h)]^3$ with $\ell\ge 1$ being an integer.
\end{theorem}

 We present in the next {\color{black} lemma} concerning  a continuous Sobolev imbedding inequality before the proof of \Cref{discrete-H1}

\begin{lemma}\label{continous-H1} There exists a positive constant $C$  such that
	\begin{align*}
	\|\bm v\|_1\le C\left(
	\|\nabla\times\bm v\|_0+\|\nabla\cdot\bm v\|_0
	\right)
	\end{align*}
	holds for any $\bm v\in [H^1_0(\Omega)]^3$.
\end{lemma}
\begin{proof}
	Since $\bm v\in [H^1_0(\Omega)]^3$, it holds $\Delta\bm v\in [H^{-1}(\Omega)]^3$. \color{black} The Poincar{\'e}'s inequality and integration by parts give that
	\begin{align*}
	\|\bm v\|_1^2\le C(\nabla\bm v,\nabla \bm v)=C(-\Delta\bm v,\bm v)\le C\|\Delta\bm v\|_{-1}\|\bm v\|_{1},
	\end{align*}
	and hence
	\begin{align}\label{l61}
	\|\bm v\|_1\le C\|\Delta \bm v\|_{-1}.
	\end{align}
 Meanwhile, utilizing integration by parts and the Cauthy-Schwarz inequality, for any $\bm w\in [H^1_0(\Omega)]^3$, we have
	\begin{align*}
	-(\color{black}\Delta\color{black}\bm v,\bm w)&=(\nabla\times(\nabla\times\bm v),\bm w)-(\color{black}\nabla\color{black}(\nabla\cdot\bm v),\bm w)\nonumber\\
	&=(\nabla\times\bm v,\nabla\times\bm w)+(\nabla\cdot\bm v,\nabla\cdot\bm w)\nonumber\\
	&\le C\left(\|\nabla\times\bm v\|_0
	+\|\nabla\cdot\bm v\|_0
	\right)\|\bm w\|_1.
	\end{align*}
	Then we arrive at
	\begin{align}\label{l62}
	\|\Delta\bm v\|_{-1}=\sup_{\bm 0\neq\bm w\in [H^1_0(\Omega)]^3}\frac{(\color{black}\Delta\color{black}\bm v,\bm w)}{\|\bm w\|_1}\le C\left(\|\nabla\times\bm v\|_0
	+\|\nabla\cdot\bm v\|_0
	\right).
	\end{align}
	The proof can be obtained immediately from the combination of both \eqref{l61} and \eqref{l62}.
\end{proof}

Now, we are in the position to prove \Cref{discrete-H1}:
\begin{proof}[Proof of \Cref{discrete-H1}:]
	By \cite[Theorem 2.2]{MR2034620}, for any $\bm v_h\in [\mathcal{P}_{\ell}(\mathcal{T}_h)]^3$ with $\ell\ge 1$, there exists an interpolation  $\bm{\mathcal I}_{h,\ell}^{\rm c}$ such that $\bm{\mathcal I}_{h,\ell}^{\rm c}\bm v_h\in [H^1_0(\Omega)]^3$, and
	\begin{align}\label{jc}
	\|\bm{\mathcal I}_{h,\ell}^{\rm c}\bm v_h-\bm v_h\|_0
	+h	\left(\sum_{K\in\mathcal{T}_h}\|\nabla(\bm{\mathcal I}_{h,\ell}^{\rm c}\bm v_h-\bm v_h)\|_{0,K}^2\right)^{\frac{1}{2}}
	\le C
	\left(
	\sum_{F\in\mathcal{E}_h}h_F\|[\![
	\bm v_h
	]\!]\|_{0,F}^2
	\right)^{\frac{1}{2}}.
	\end{align}
	Using \Cref{continous-H1}, the triangle inequality and the estimate in \eqref{jc}, we obtain that
	\begin{align*}
	\|\bm{\mathcal I}_{h,\ell}^{\rm c}\bm v_h\|_1^2&\le
	C(\|\nabla\times\bm{\mathcal I}_{h,\ell}^{\rm c}\bm v_h\|_0^2+\|\nabla\cdot\bm{\mathcal I}_{h,\ell}^{\rm c}\bm v_h\|_0^2)\nonumber\\
	&\le C
	\sum_{K\in\mathcal{T}_h}
	\left(
	\|\nabla\times\bm v_h\|_{0,K}^2+\|\nabla\cdot\bm v_h\|_{0,K}^2
	+\|\nabla(\bm{\mathcal I}_{h,\ell}^{\rm c}\bm v_h-\bm v_h   )\|_{0,K}^2\right) \nonumber\\
	&\le  C
	\left[
	\sum_{K\in\mathcal{T}_h}\left(
	\|\nabla\times\bm v_h\|_{0,K}^2+\|\nabla\cdot\bm v_h\|_{0,K}^2
	\right)+\sum_{F\in\mathcal{E}_h}h_F^{-1}\|[\![\bm v_h]\!]\|^2_{0,F}\right].
	\end{align*}
	Due to the above estimate, the triangle inequality and the estimate in \eqref{jc}, the proof of the first inequality could be obtained immediately. The second inequality could be obtained through the combination of the discrete Sobolev imbedding inequality in \cite[Theorem 2.1]{MR2629994} and the first inequality.
\end{proof}

\begin{lemma}[{\cite[Theorem 3.1]{2019-weifeng}}]\label{th:embed} For any $\bm v_h\in [\mathcal{P}_{k+1}(\mathcal{T}_h)]^3$ satisfying
	\begin{align*}
	(\bm v_h,\nabla q_h)=0\, , \qquad\forall\, q_h\in Q_h= H_0^1(\Omega)\cap \mathcal{P}_{k+2}(\mathcal{T}_h),
	\end{align*}
	there exists a positive constant $C$ such that
	\begin{align*}
	\|\bm v_h\|_{0,3}\le C
	\left(
	\sum_{K\in\mathcal{T}_h}\|\nabla\times\bm v_h\|_{0,K}^2
	+\sum_{ F\in\mathcal{E}_h}\color{black} h_F^{-1}\color{black}\|\bm n\times[\![\bm v_h]\!]\|^2_{0,F}
	\right)^{\frac{1}{2}}.
	\end{align*}
\end{lemma}

With the above lemma, we can derive discrete Sobolev inequalities for $\bm H({\rm curl})$-conforming functions.
\begin{lemma} \label{th:embed2} For any $\bm v_h\in \bm E_h=\bm H_0({\rm curl};\Omega)\cap [\mathcal P_{k+1}(\mathcal{T}_h)]^3$, if there holds
	\begin{align}\label{conditon}
	(\bm v_h,\nabla q_h)=0\, ,\qquad\forall\, q_h\in Q_h= H_0^1(\Omega)\cap \mathcal{P}_{k+2}(\mathcal{T}_h),
	\end{align}
	then we have
	\begin{align}\label{L31}
	\|\bm v_h\|_{0,3}\le C\|\nabla\times\bm v_h\|_{0},
	\end{align}
	and
	\begin{align}\label{L32}
	\|\nabla\times\bm v_h\|_{0,3}\le C
	\left(
	\sum_{K\in\mathcal{T}_h}\|\nabla\times(\nabla\times\bm v_h)\|_{0,K}^2
	+\sum_{E\in\mathcal{F}_h}\color{black}h_F^{-1}\color{black}\|\bm n\times[\![\nabla\times\bm v_h]\!]\|^2_{0,F}
	\right)^{\frac{1}{2}}.
	\end{align}
\end{lemma}
\begin{proof} The result \eqref{L31} follows from \eqref{conditon},
	\Cref{th:embed}, and the fact that $\bm v_h\in \bm H_0({\rm curl};\Omega)$.  Since $\nabla\cdot(\nabla\times\bm v_h)=0$ holds for any $q_h\in Q_h$, we have that
	\begin{align*}
	(\nabla\times\bm v_h,\nabla q_h)=-(\nabla\cdot(\nabla\times\bm v_h), q_h)=0
	\end{align*}
	from which the result \eqref{L32} follows immediately because of \Cref{th:embed}.
\end{proof}

\color{black}
\subsection{Existence of a unique solution\label{s42}}

In this subsection, we   state the existence of a unique solution of \eqref{FEM}.
\begin{theorem} The variational equation \eqref{FEM} admits a unique solution  $(\bm u_h,p_h)\in \bm E_h\times Q_h$ as $\tau$ is sufficiently large.
\end{theorem}
\begin{proof}
	We define a mesh-dependent norm as
	\begin{align*}
	\interleave\bm v_h \interleave^2&:=\|\bm v_h\|^2_0
	+\|\nabla\times\bm v_h\|^2_0
	+\sum_{K\in\mathcal{T}_h}\|\nabla\times(\nabla\times\bm v_h)\|^2_{0,K}+\sum_{F\in\mathcal{E}_h}\tau h_F^{-1}\|\bm n\times[\![\nabla\times\bm v_h]\!]\|^2_{0,F}.
	\end{align*}
	For all $\bm u_h\in \bm E_h$ satisfying
	\begin{align*}
	(\bm u_h,\nabla q_h)=0\, , \qquad \forall \,q_h\in Q_h,
	\end{align*}
	 there holds, due to  \Cref{th:embed2},
	\begin{align}\label{ellip1}
	\interleave \bm u_h \interleave^2&\le
	C\left(\sum_{K\in\mathcal{T}_h}\|\nabla\times(\nabla\times\bm u_h)\|^2_{0,K}
	+\sum_{F\in\mathcal{E}_h}\tau h_F^{-1}\|\bm n\times[\![\nabla\times\bm u_h]\!]\|^2_{0,F}\right),
	\end{align}
	and
	\begin{align}\label{ellip2}
	&
	\sum_{K\in\mathcal{T}_h}(\nabla\times(\nabla\times\bm u_h),\nabla\times(\nabla\times\bm u_h))_{K}\nonumber\\
	&\quad-2\sum_{F\in\mathcal{E}_h}\langle\{\!\!\{ \nabla\times(\nabla\times\bm u_h)\}\!\!\},\bm n\times[\![\nabla\times\bm u_h]\!] \rangle_F\nonumber\\
	&\quad
	+\sum_{F\in\mathcal{E}_h}\tau h_F^{-1}\langle\bm n\times[\![\nabla\times\bm u_h]\!],\bm n\times[\![\nabla\times\bm u_h]\!] \rangle_F\nonumber\\
	&\ge \sum_{K\in\mathcal{T}_h}\|\nabla\times(\nabla\times\bm u_h)\|^2_{0,K}
	-C\sum_{K\in\mathcal T_h}h_K^{-\frac{1}{2}}\|\nabla\times(\nabla\times\bm u_h)\|_{0,K}
	\|\bm n\times[\![\nabla\times\bm u_h]\!]\|_{0,\partial K}\nonumber\\
	&\quad+\sum_{F\in\mathcal{E}_h}\tau h_F^{-1}\|\bm n\times[\![\nabla\times\bm u_h]\!]\|^2_{0,F}\nonumber\\
	&\ge C\left(\sum_{K\in\mathcal{T}_h}\|\nabla\times(\nabla\times\bm u_h)\|^2_{0,K}
	+\sum_{F\in\mathcal{E}_h}\tau h_F^{-1}\|\bm n\times[\![\nabla\times\bm u_h]\!]\|^2_{0,F}\right).
	\end{align}
	Here, $\tau$ is a sufficiently large constant. 	By taking $\bm v_h=\nabla q_h\in\bm E_h$ and noticing $\nabla\times(\nabla q_h)=0$, we get
	\begin{align}\label{LBB}
	\sup_{\bm 0\neq \bm v_h\in \bm E_h}\frac{(\bm v_h,\nabla q_h)}{ \interleave\bm v_h \interleave
	}\ge \|\nabla q_h\|_0.
	\end{align}
Then the existence of a unique solution is followed by \eqref{ellip1}, \eqref{ellip2}, \eqref{LBB}, and the theory for the mixed problem \cite[Theorem 1.1]{MR1115205}.
\end{proof}

\color{black}

\subsection{$L^3$ stability of $u_h$ and $\nabla\times u_h$\label{s43}}

\begin{lemma}\label{stable0} 	Let $(\bm u_h,p_h)\color{black}\in \bm E_h\times Q_h$ be the solution of \eqref{FEM}, then
	when the constant $\tau>0$ is sufficient large, we have
	\begin{align*}
	\sum_{K\in\mathcal{T}_h}\|\nabla\times(\nabla\times\bm u_h)\|^2_{0,K}
	+\sum_{F\in\mathcal{E}_h}\color{black}\tau h_F^{-1}\color{black}\|\bm n\times[\![\nabla\times\bm u_h]\!]\|^2_{0,F}
	\le C(\bm f,\bm u_h).
	\end{align*}
\end{lemma}
\begin{proof}
	We take $\bm v_h=\bm u_h\in \bm E_h$ in \eqref{fem01}, and $q_h=p_h\in Q_h$ in \eqref{fem02}, and combine the corresponding  equalities to obtain
	\begin{align}\label{PP1}
	\sum_{K\in\mathcal{T}_h}\|\nabla\times(\nabla\times\bm u_h)\|^2_{0,K}
	-2\sum_{F\in\mathcal{E}_h}\langle\{\!\!\{ \nabla\times(\nabla\times\bm u_h)\}\!\!\},\bm n\times[\![\nabla\times\bm u_h]\!] \rangle_F&\nonumber\\
	+\sum_{F\in\mathcal{E}_h}\color{black}\tau h_F^{-1}\color{black}\|\bm n\times[\![\nabla\times\bm u_h]\!]\|^2_{0,F}&=(\bm f,\bm u_h).
	\end{align}
	Since $\tau>0$ is a sufficient large constant, we have
	\begin{align}\label{PP2}
	&\left|
	2\sum_{F\in\mathcal{E}_h}\langle\{\!\!\{ \nabla\times(\nabla\times\bm u_h)\}\!\!\},\bm n\times[\![\nabla\times\bm u_h]\!] \rangle_F
	\right|\nonumber\\
	&\qquad\le 2\sum_{F\in\mathcal{E}_h}\|\{\!\!\{ \nabla\times(\nabla\times\bm u_h)\}\!\!\}\|_{0,F}
	\|\bm n\times[\![\nabla\times\bm u_h]\!]\|_{0,F} \nonumber\\
	&\qquad\le
	\frac{1}{2}\sum_{K\in\mathcal{T}_h}\|\nabla\times(\nabla\times\bm u_h)\|^2_{0,K}
	+\frac{1}{2}\sum_{F\in\mathcal{E}_h}\color{black}\tau h_F^{-1}\color{black}\|\bm n\times[\![\nabla\times\bm u_h]\!]\|^2_{0,F}.
	\end{align}
	Combining \eqref{PP1} and \eqref{PP2} yields
	\begin{align*}
	\sum_{K\in\mathcal{T}_h}\|\nabla\times(\nabla\times\bm u_h)\|^2_{0,K}
	+\sum_{F\in\mathcal{E}_h}\color{black}\tau h_F^{-1}\color{black}\|\bm n\times[\![\nabla\times\bm u_h]\!]\|^2_{0,F}
	\le C(\bm f,\bm u_h).
	\end{align*}
\end{proof}

With the above lemma, we are ready to prove the following stability result.

\begin{theorem} \label{th:stability}
	Let $(\bm u_h,p_h)\color{black}\in \bm E_h\times Q_h$ be the solution of \eqref{FEM}, then
	when the constant $\tau>0$ is  sufficient large, we have the following stability result
	\begin{align*}
	&\|\bm u_h\|_{0,3}+\|\nabla\times \bm u_h\|_{0,3}
	\nonumber\\
	&+\left(
	\sum_{K\in\mathcal{T}_h}\|\nabla\times(\nabla\times\bm u_h)\|^2_{0,K}
	+\sum_{F\in\mathcal{E}_h}\color{black}\tau h_F^{-1}\color{black}\|\bm n\times[\![\nabla\times\bm u_h]\!]\|^2_{0,F}
	\right)^{\frac{1}{2}}
	\le C\|\bm f\|_{0,\frac{3}{2}}.
	\end{align*}
\end{theorem}

\begin{proof}
	
	Using \Cref{th:embed2}, one obtains
	\begin{align}\label{PP4}
	\|\bm u_h\|_{0,3}\le C\|\nabla\times\bm u_h\|_{0}\le C\|\nabla\times\bm u_h\|_{0,3},
	\end{align}
	and
	\begin{align}\label{PP3}
	\|\nabla\times\bm u_h\|_{0,3}\le
	C\left(
	\sum_{K\in\mathcal{T}_h}\|\nabla\times(\nabla\times\bm u_h)\|^2_{0,K}
	+\sum_{F\in\mathcal{E}_h}\color{black}\tau h_F^{-1}\color{black}\|\bm n\times[\![\nabla\times\bm u_h]\!]\|^2_{0,F}
	\right)^{\frac{1}{2}}.
	\end{align}
Therefore, applying \eqref{PP4}, \eqref{PP3}, \Cref{stable0} and the Cauthy-Schwarz inequality, we arrive at
	\begin{align*}
	&\|\bm u_h\|_{0,3}+\|\nabla\times \bm u_h\|_{0,3}+\left(
	\sum_{K\in\mathcal{T}_h}\|\nabla\times(\nabla\times\bm u_h)\|^2_{0,K}
	+\sum_{F\in\mathcal{E}_h}\color{black}\tau h_F^{-1}\color{black}\|\bm n\times[\![\nabla\times\bm u_h]\!]\|^2_{0,F}
	\right)^{\frac{1}{2}}\\ \nonumber
	&\le C\|\bm f\|_{0,\frac{3}{2}}.
	\end{align*}
\end{proof}

\subsection{ $L^6$ and discrete $H^1$ stability of $\nabla\times u_h$ \label{s44}}

\begin{theorem} \label{th:stability2}
	Let $(\bm u_h,p_h)\color{black}\in \bm E_h\times Q_h$ be the solution of \eqref{FEM},
	when $\tau>0$ is a sufficient large constant, then we have the following stability result
	\begin{align*}
	\|\nabla\times\bm u_h\|_{0,6}^2+
	\sum_{K\in\mathcal{T}_h}\|\nabla(\nabla\times\bm u_h)\|_{0,K}^2
	+\sum_{F\in\mathcal{E}_h}h_F^{-1}\|[\![\nabla\times\bm u_h]\!]\|_{0,F}^2&\nonumber\\
	+
	\sum_{K\in\mathcal{T}_h}\|\nabla\times(\nabla\times\bm u_h)\|_{0,K}^2
	+\sum_{F\in\mathcal{E}_h}h_F^{-1}\|\bm n\times[\![\nabla\times\bm u_h]\!]\|^2_{0,F}	
	&\le  C\|\bm f\|_{0,\frac{3}{2}}^2.
	\end{align*}
\end{theorem}
\begin{proof} Since $\bm u_h\in \bm H_0({\rm curl};\Omega)$,  there holds that $\nabla\times\bm u_h\in\bm H({\rm div};\Omega)$, and $\nabla\cdot(\nabla\times\bm u_h)=0$. Using the triangle inequality, one arrives at
	\begin{align*}
	\|[\![\nabla\times\bm u_h]\!]\|_{0,F}
	&=
	\|(\bm n\times[\![\nabla\times\bm u_h]\!])\times\bm n+(\bm n\cdot[\![\nabla\times\bm u_h]\!])\color{black}\bm n\color{black}\|_{0,F}
	\nonumber\\
	&\le
	\|(\bm n\times[\![\nabla\times\bm u_h]\!])\times\bm n\|_{0,F}
	+	\|(\bm n\cdot[\![\nabla\times\bm u_h]\!])\color{black}\bm n\color{black}\|_{0,F}
	\nonumber\\
	&\le C
	\|\bm n\times[\![\nabla\times\bm u_h]\!]\|_{0,F}.
	\end{align*}
Making use of the above estimate, \Cref{discrete-H1} with $\bm v_h=\nabla\times\bm u_h$, and \Cref{stable0}, we obtain
	\begin{align*}
	&
	\|\nabla\times\bm u_h\|_{0,6}^2+
	\sum_{K\in\mathcal{T}_h}\|\nabla(\nabla\times\bm u_h)\|_{0,K}^2
	+\sum_{F\in\mathcal{E}_h}h_F^{-1}\|[\![\nabla\times\bm u_h]\!]\|_{0,F}^2 \nonumber\\
	&\quad\le 	\sum_{K\in\mathcal{T}_h}\|\nabla(\nabla\times\bm u_h)\|_{0,K}^2
	+\sum_{F\in\mathcal{E}_h}h_F^{-1}\|[\![\nabla\times\bm u_h]\!]\|_{0,F}^2 \nonumber\\
	&\quad\le	C\left[
	\sum_{K\in\mathcal{T}_h}\left(
	\|\nabla\times(\nabla\times\bm u_h)\|^2_{0,K}
	+\|\nabla\cdot(\nabla\times\bm u_h)\|^2_{0,K}
	\right)
	+\sum_{F\in\mathcal{E}_h}h_F^{-1}\|[\![\nabla\times\bm u_h]\!]\|^2_{0,F}
	\right]\nonumber\\
	&\quad=	C\left[
	\sum_{K\in\mathcal{T}_h}
	\|\nabla\times(\nabla\times\bm u_h)\|^2_{0,K}
	+\sum_{F\in\mathcal{E}_h}h_F^{-1}\|\bm n\times[\![\nabla\times\bm u_h]\!]\|^2_{0,F}
	\right]\nonumber\\
	&\quad\le C(\bm f,\bm u_h)\nonumber\\
	&\quad \le C\|\bm f\|_{0,\frac{3}{2}}\|\bm u_h\|_{0,3}~.
	\end{align*}
Meanwhile, it holds from \Cref{th:embed2}
	\begin{align*}
	\|\bm u_h\|_{0,3}&\le C\|\nabla\times\bm u_h\|_{0}\le C\|\nabla\times\bm u_h\|_{0,3}\nonumber\\
	&\le C\left(
	\sum_{K\in\mathcal{T}_h}\|\nabla(\nabla\times\bm u_h)\|_{0,K}^2
	+\sum_{F\in\mathcal{E}_h}h_F^{-1}\|[\![\nabla\times\bm u_h]\!]\|_{0,F}^2
	\right)^{\frac{1}{2}}.
	\end{align*}
 The combination of the above two inequities completes the proof of the theorem.

\end{proof}

\section{Main error estimates\label{sec:error}}

We first give formulas for integration by parts on low regularity functions.

\begin{lemma}[{\cite[Theorem 1.4.4.6]{MR775683}}]\label{aux1} If $\phi\in H^{\frac{1}{2}+\delta}(\Omega)$ with $\delta\in (0,1/2]$, then
	$\frac{\partial\phi}{\partial x_i}\in H^{-\frac{1}{2}+\delta}(\Omega)$ for $i=1,2,3$, and there exists a constant $C>0$ such that
	\begin{align*}
	\left\| \frac{\partial\phi}{\partial x_i}\right\|_{-\frac{1}{2}+\delta}\le
	C\| \phi\|_{\frac{1}{2}+\delta}~.
	\end{align*}	
\end{lemma}

\begin{theorem}\label{th-int1} If $\nabla\times(\nabla\times\bm u)\in [H^{\frac{1}{2}+\delta}(\Omega)]^3$ with $\delta\in (0,\frac{1}{2}]$, $\nabla\times(\nabla\times(\nabla\times(\nabla\times\bm u)))\in [L^2(\Omega)]^3$, and $\bm v_h\in \bm E_h$, then it holds
	\begin{align}\label{int}
	(\nabla\times(\nabla\times(\nabla\times(\nabla\times\bm u))),\bm v_h)
	=(\nabla\times(\nabla\times(\nabla\times\bm u)),\nabla\times\bm v_h)~,
	\end{align}
	where $(\nabla\times(\nabla\times(\nabla\times\bm u)),\nabla\times\bm v_h)$ is regarded as a duality pairing between $H^{-\frac{1}{2} + \delta}(\Omega)$ and $H_{0}^{\frac{1}{2} - \delta}(\Omega)$.
\end{theorem}
\begin{proof}
	For all $\bm v_h\in \bm E_h$, it is easy to see that $\bm v_h\in [H^{\frac{1}{2}-\delta}(\Omega)]^3$, and $\nabla\times\bm v_h\in [H^{\frac{1}{2}-\delta}(\Omega)]^3$. Since $\nabla\times(\nabla\times\bm u)\in [H^{\frac{1}{2}+\delta}(\Omega)]^3$,
	according to \Cref{aux1},
	it holds $\nabla\times(\nabla\times(\nabla\times\bm u))\in [H^{-\frac{1}{2}+\delta}(\Omega)]^3$.

 Moreover, since  $H^{\frac{1}{2}-\delta}(\Omega)=H_0^{\frac{1}{2}-\delta}(\Omega)$ for $\delta\in (0,\frac{1}{2}]$ (\cite[Theorem
	1.4.2.4]{MR775683} or \cite[Theorem 3.40]{MR1742312}),  $(\nabla\times(\nabla\times(\nabla\times\bm u)),\nabla\times\bm v_h)$ can
	be regarded as a duality pairing between $[H^{-\frac{1}{2}+\delta}(\Omega)]^3$ and $[H_0^{\frac{1}{2}-\delta}(\Omega)]^3$.
	
Next, we are going to illustrate that $[\mathcal C_0^{\infty}(\Omega)]^3$ is dense in $\bm E_h$ with respect to the norm $\|\cdot\|_{\bm H^{\frac{1}{2}-\delta}({\rm curl};\Omega)}$.  Let $\bm v_h\in \bm E_h$, we define $\tilde{\bm v}_{h}$ on $\mathbb R^{3}$ such that it is identical to $\bm v_{h}$ in $\Omega$, and is trivial outside $\Omega$. Obviously, $\tilde{\bm v}_{h}$ belongs to $\bm H({\rm curl}, \mathbb R^{3})$.  It can also be shown that both $\tilde{\bm v}_{h}$ and $\nabla\times\tilde{\bm v}_{h}$ belong to $[H^{\frac{1}{2} - \delta}(\mathbb R^3)]^{3}$, where $\delta\in(0,\frac{1}{2}]$.
As a result, we have,  similar to the proof of part (ii) of \cite[Theorem 3.29]{MR1742312}, that there exists a sequence of vector fields in $[\mathcal C_{0}^{\infty}(\Omega)]^{3}$  which converges to $\bm v_{h}$ with respect to the norm $\|\cdot\|_{\bm H^{\frac{1}{2}-\delta}({\rm curl};\Omega)}$.
	
Finally, the equation \eqref{int} follows immediately from the standard duality argument and the fact that
	\begin{align*}
	(\nabla\times(\nabla\times(\nabla\times(\nabla\times\bm u))),\bm v)
	&=((\nabla\times(\nabla\times(\nabla\times\bm u)),\nabla\times\bm v)&\forall \,\bm v\in \mathcal [C_0^{\infty}(\Omega)]^3.
	\end{align*}
\end{proof}

\begin{theorem}\label{th-int2} For all $\phi\in H^{\frac{1}{2}+\delta}(\Omega)$ with $\delta\in (0,\frac{1}{2}]$, and $\psi_h\in \mathcal P_{\ell}(\mathcal T_h)$ with integer $\ell\ge 0$,   we have
	\begin{align}\label{33}
	\left(\frac{\partial\phi}{\partial x_i},\psi_h\right)=-
	\sum_{K\in\mathcal T_h}\left(\phi,\frac{\partial\psi_h}{\partial x_i}\right)_K
	+\sum_{K\in\mathcal T_h}\langle\phi,\psi_h n_i \rangle_{\partial K}
	\end{align}
	for $i=1,2,3$, where $	\left(\frac{\partial\phi}{\partial x_i},\psi_h\right)$ denotes  a duality pairing between $H^{-\frac{1}{2} + \delta}(\Omega)$ and $H_{0}^{\frac{1}{2} - \delta}(\Omega)$~.
\end{theorem}
\begin{proof} According to part (i) of \cite[Theorem 3.29]{MR1742312}, $\mathcal D(\bar{\Omega})$ is dense in $H^{\frac{1}{2}+\delta}(\Omega)$, where
	\begin{align*}
	\mathcal D(\bar{\Omega})=\{v: v=\tilde{v}|_{\Omega} \text{ for some }\tilde v\in \mathcal C^{\infty}(\mathbb R^3)\}.
	\end{align*}
	We choose $\{\phi_n\}_{n=1}^{\infty}\subset \mathcal D(\bar{\Omega})$ such that
	\begin{align}\label{tozero}
	\|\phi_n-\phi\|_{\frac{1}{2}+\delta}\to 0 \text{ as }n\to\infty.
	\end{align}
	According to \Cref{aux1}, we have
	\begin{align*}
	\left\|\frac{\partial \phi_n}{\partial x_i}-\frac{\partial\phi}{\partial x_i}\right\|_{-\frac{1}{2}+\delta}\to 0 \text{ as }n\to\infty.
	\end{align*}
	For any positive integer $n$, we have
	\begin{align*}
	\left(\frac{\partial\phi_n}{\partial x_i},\psi_h\right)=-
	\sum_{K\in\mathcal T_h}\left(\phi_n,\frac{\partial\psi_h}{\partial x_i}\right)_K
	+\sum_{K\in\mathcal T_h}\langle\phi_n,\psi_h n_i \rangle_{\partial K}.
	\end{align*}
	Note that \eqref{tozero} implies that
	\begin{align*}
	\sum_{K\in\mathcal T_h}\|\phi_n-\phi\|^2_{L^2(\partial K)}\to 0 \text{ as }n\to\infty~,
	\end{align*}
	$\psi_h\in \mathcal P_{\ell}(\mathcal T_h)$ and $\delta\in (0,\frac{1}{2}]$, we have $\psi_h\in H^{\frac{1}{2}-\delta}(\Omega)=H_0^{\frac{1}{2}-\delta}(\Omega)$, so
	$\left(\frac{\partial\phi}{\partial x_i},\psi_h\right)$ can be regarded as  a duality pairing between  $H^{-\frac{1}{2}+\delta}(\Omega)$ and $H_0^{\frac{1}{2}-\delta}(\Omega)$. Then it holds
	\begin{align*}
	\left|\left(\frac{\partial\phi}{\partial x_i},\psi_h\right)
	-\left(\frac{\partial\phi_n}{\partial x_i},\psi_h\right)\right|\le C
	\left\|\frac{\partial\phi}{\partial x_i}-\frac{\partial\phi_n}{\partial x_i}\right\|_{-\frac{1}{2}+\delta}\|\psi_h\|_{\frac{1}{2}-\delta}~,
	\end{align*}
and  \eqref{33} follows immediately.
\end{proof}
\color{black}

\begin{assumption} \label{ass} In the rest part of this paper, we will assume that the following regularities hold true for the weak solution $(\bm u,p)\in \bm H_0({\rm curl}^2;\Omega)\times H_0^1(\Omega)$ of \eqref{PDE:orignial}:
	\begin{align*}
	&\bm u\in [H^{r_{u_0}}(\Omega)]^3,\quad
	\nabla\times \bm u            \in [H^{r_{u_1}}(\Omega)]^3,\quad
	\nabla\times(\nabla\times \bm u)\in [H^{r_{u_2}}(\Omega)]^3,\quad
	p\in H^{r_p}(\Omega)~,
	\end{align*}
	where
	$r_{u_0}\in (\frac{1}{2},\infty)$,
	$r_{u_1}\in [1,\infty)$,
	$r_{u_2}\in (\frac{1}{2},\infty)$,
	and
	$r_{p}\in (\frac{3}{2},\infty)$.
	We also assume that $r_{u_0}\le r_{u_1}\le r_{u_0}+1$  in order to simplify the notations in the error analysis.
\end{assumption}

{\color{black}
We notice that the above assumption holds with $r_{u_0}=r_{u_2}=r_p=2$ when $\Omega$ is convex.
}

\begin{lemma}\label{pipi} Let \color{black} $(\bm u,p)\in \bm H_0({\rm curl}^2;\Omega)\times H_0^1(\Omega)$ be the weak solution of \eqref{PDE:orignial} and \Cref{ass} holds true, then
	\begin{subequations}\label{pi}
		\begin{align}
		\sum_{K\in\mathcal{T}_h}(\nabla\times(\nabla\times\bm u),\nabla\times(\nabla\times\bm v_h))_{K}+(\nabla p,\bm v_h)&\nonumber\\
		-\sum_{F\in\mathcal{E}_h}\langle\{\!\!\{ \nabla\times(\nabla\times\bm u)\}\!\!\},\bm n\times[\![\nabla\times\bm v_h]\!] \rangle_F&\nonumber\\
		-\sum_{F\in\mathcal{E}_h}\langle\{\!\!\{ \nabla\times(\nabla\times\bm v_h)\}\!\!\},\bm n\times[\![\nabla\times\bm u]\!] \rangle_F&\nonumber\\
		+\sum_{F\in\mathcal{E}_h}\color{black}\tau h_F^{-1}\color{black}\langle\bm n\times[\![\nabla\times\bm u]\!],\bm n\times[\![\nabla\times\bm v_h]\!] \rangle_F
		&=(\bm f,\bm v_h),\label{pi1}\\
		(\bm u,\nabla q_h)&=0\label{pi2}
		\end{align}
	are true	 for all $ (\bm v_h,q_h)\in  \bm E_h\times Q_h$.
	\end{subequations}
\end{lemma}

\begin{proof}  For any $ \bm v_h\in \bm E_h$, \color{black} since the weak solution $(\bm u,p)\in \bm H_0({\rm curl}^2;\Omega)\times H_0^1(\Omega)$ satisfies \eqref{o1} in the distribution sense,
by \Cref{th:reg}, we have  $\nabla\times(\nabla\times(\nabla\times(\nabla\times\bm u))) \in [L^2(\Omega)]^3$, which further leads to, together with the fact that $ \nabla p, \bm f\in [L^2(\Omega)]^3$,
	\begin{align*}
	(\nabla\times(\nabla\times(\nabla\times(\nabla\times\bm u))),\bm v_h)+(\nabla p,\bm v_h)=(\bm f,\bm v_h).
	\end{align*}
	From \Cref{th-int1}\color{black}, we have
	\begin{align*}
	(\nabla\times(\nabla\times(\nabla\times\bm u)),\nabla\times\bm v_h)+(\nabla p,\bm v_h)=(\bm f,\bm v_h).
	\end{align*}
	\color{black}
	According to \Cref{th-int2}\color{black}, we have

	\begin{align*}
	&\sum_{K\in\mathcal{T}_h}(\nabla\times(\nabla\times\bm u),\nabla\times(\nabla\times\bm v_h))_K\nonumber\\
	&\qquad-\sum_{F\in\mathcal{E}_h}\langle
	\{\!\!\{\nabla\times(\nabla\times\bm u)\}\!\!\},\bm n\times[\![\nabla\times\bm v_h]\!]
	\rangle_F
	+(\nabla p,\bm v_h)=(\bm f,\bm v_h),
	\end{align*}
	where we have utilized  the fact  $\nabla\times(\nabla\times\bm u)=\{\!\!\{\nabla\times(\nabla\times\bm u)\}\!\!\}$   because of  $\nabla\times(\nabla\times\bm u)\in [H^{s_{u_2}}(\Omega)]^3$ with $s_{u_2}>\frac{1}{2}$  .  Moreover, since $\color{black}\bm n\times[\![\nabla\times\bm u]\!]=\bm 0$ because of $\nabla\times\bm u\in[H^{s_{u_1}}(\Omega)]^3$ with $s_{u_1}\ge 1$ \color{black} and $\bm n\times(\nabla\times\bm u)|_{\partial\Omega}=\bm 0$,  \eqref{pi1} follows immediately. Finally, \eqref{pi2}  can be obtained due to the fact $Q_h\in H_0^1(\Omega)$.
\end{proof}

\begin{theorem}\label{error estimates} 	\color{black} Let $(\bm u,p)\in \bm H_0({\rm curl}^2;\Omega)\times H_0^1(\Omega)$ be the weak solution of \eqref{PDE:orignial}  and \Cref{ass} holds true; let $(\bm u_h, p_h)\in \bm E_h\times Q_h$ be the solution of \eqref{FEM}. \color{black} Then we have the following error estimates
	\begin{align*}
	&\|\bm u-\bm u_h\|_0+\|\nabla\times(\bm u-\bm u_h)\|_0\nonumber\\
	&\quad+\left(\sum_{K\in\mathcal{T}_h}\|\nabla\times(\nabla\times (\bm u-\bm u_h))\|^2_{0,K}
	+\sum_{F\in\mathcal{E}_h}\color{black}\tau h_F^{-1}\color{black}\|\bm n\times[\![\nabla\times (\bm u-\bm u_h)]\!]\|^2_{0,F}\right)^{\frac{1}{2}}
	\nonumber\\
	&\quad\quad	
	\le C
	(
	h^{s_{u_0}}\|\bm u\|_{s_{u_0}}
	+	h^{s_{u_1}-1}\|\nabla\times\bm u\|_{s_{u_1}}\color{black}
	+   h^{s_{u_2}}\|\nabla\times(\nabla\times\bm u)\|_{s_{u_2}}
	+	h^{s_p-1}\|p\|_{s_p}
	),\\
	& \|\nabla (p- p_h)\|_0\le Ch^{s_p-1}\|p\|_{s_p},
	\end{align*}
	and the discrete  $H^1$ norm error estimate for $\nabla\times(\bm u-\bm u_h)$
	\begin{align*}
	&\left(\sum_{K\in\mathcal{T}_h}\|\nabla(\nabla\times (\bm u-\bm u_h))\|^2_{0,K}
	+\sum_{F\in\mathcal{E}_h}h_F^{-1}\|[\![\nabla\times (\bm u-\bm u_h)]\!]\|^2_{0,F}\right)^{\frac{1}{2}}
	\nonumber\\
	&\quad\quad	
	\le C
	(
	h^{s_{u_0}}\|\bm u\|_{s_{u_0}}
	+	h^{s_{u_1}-1}\|\nabla\times\bm u\|_{s_{u_1}}
	\color{black}
	+   h^{s_{u_2}}\|\nabla\times(\nabla\times\bm u)\|_{s_{u_2}}
	+	h^{s_p-1}\|p\|_{s_p}
	),
	\end{align*}
	where
	$s_{u_0}\in(\frac{1}{2},\min(r_{u_0},\color{black}k+2)]$,
	$s_{u_1}\in [1,\min(r_{u_1},k+1)]$,
	$s_{u_2}\in [1,\min(r_{u_2},k+1)]$, and
	$s_{p}  \in(\frac{3}{2},\min(r_{p},k+3)]$.
\end{theorem}

\begin{proof}
	We subtract \eqref{FEM} from \eqref{pi} to get: for all $(\bm v_h,q_h)\in \bm E_h\times Q_h$, there hold
	\begin{subequations}\label{error1}
		\begin{align}
		\sum_{K\in\mathcal{T}_h}(\nabla\times(\nabla\times(\bm u-\bm u_h)),\nabla\times(\nabla\times\bm v_h))_{K}+(\nabla (p-p_h),\bm v_h)&\nonumber\\
		-\sum_{F\in\mathcal{E}_h}\langle\{\!\!\{ \nabla\times(\nabla\times(\bm u-\bm u_h))\}\!\!\},\bm n\times[\![\nabla\times\bm v_h]\!] \rangle_F&\nonumber\\
		-\sum_{F\in\mathcal{E}_h}\langle\{\!\!\{ \nabla\times(\nabla\times\bm v_h)\}\!\!\},\bm n\times[\![\nabla\times(\bm u-\bm u_h)]\!] \rangle_F&\nonumber\\
		+\sum_{F\in\mathcal{E}_h}\color{black}\tau h_F^{-1}\color{black}\langle\bm n\times[\![\nabla\times(\bm u-\bm u_h)]\!],\bm n\times[\![\nabla\times\bm v_h]\!] \rangle_F
		&=0,\label{error11}\\
		(\bm u-\bm u_h,\nabla q_h)&=0.\label{error12}
		\end{align}
	\end{subequations}
	To simplify the notation, we define
	\begin{align*}
	e_h^{\bm u}:=\bm{\Pi}_h^{\bm E}\bm u-\bm u_h,\qquad
	e_h^{p}:={\Pi}_h^{Q}p-p_h.
	\end{align*}
	By taking $\bm v_h=e_h^{\bm u}\in \bm E_h$ in \eqref{error11} and $q_h=e_h^p\in Q_h$ in \eqref{error12}, we can get
	\begin{subequations}\label{error2}
		\begin{align}
		\sum_{K\in\mathcal{T}_h}(\nabla\times(\nabla\times(\bm u-\bm u_h)),\nabla\times(\nabla\times e_h^{\bm u}))_{K}+(\nabla (p-p_h), e_h^{\bm u})&\nonumber\\
		-\sum_{F\in\mathcal{E}_h}\langle\{\!\!\{ \nabla\times(\nabla\times(\bm u-\bm u_h))\}\!\!\},\bm n\times[\![\nabla\times e_h^{\bm u}]\!] \rangle_F&\nonumber\\
		-\sum_{F\in\mathcal{E}_h}\langle\{\!\!\{ \nabla\times(\nabla\times e_h^{\bm u})\}\!\!\},\bm n\times[\![\nabla\times(\bm u-\bm u_h)]\!] \rangle_F&\nonumber\\
		+\sum_{F\in\mathcal{E}_h}\color{black}\tau h_F^{-1}\color{black}\langle\bm n\times[\![\nabla\times(\bm u-\bm u_h)]\!],\bm n\times[\![\nabla\times e_h^{\bm u}]\!] \rangle_F
		&=0,\label{error21}\\
		(\bm u-\bm u_h,\nabla e_h^p)&=0.\label{error22}
		\end{align}
	\end{subequations}
Reformulating  \eqref{error2} and noticing that $(\bm{\Pi}_h^{\bm E}\bm u-\bm u,\nabla e_h^p)=0$ from \eqref{or-p}, we obtain that
	\begin{subequations}\label{error3}
		\begin{align}
		&\sum_{K\in\mathcal{T}_h}(\nabla\times(\nabla\times e_h^{\bm u}),\nabla\times(\nabla\times e_h^{\bm u}))_{K}+(\nabla e_h^p, e_h^{\bm u})\nonumber\\
		&\quad-2\sum_{F\in\mathcal{E}_h}\langle\{\!\!\{ \nabla\times(\nabla\times e_h^{\bm u})\}\!\!\},\bm n\times[\![\nabla\times e_h^{\bm u}]\!] \rangle_F\nonumber\\
		&\quad+\sum_{F\in\mathcal{E}_h}\color{black}\tau h_F^{-1}\color{black}\langle\bm n\times[\![\nabla\times e_h^{\bm u}]\!],\bm n\times[\![\nabla\times e_h^{\bm u}]\!] \rangle_F
		\nonumber\\
		&\quad\quad=	\sum_{K\in\mathcal{T}_h}(\nabla\times(\nabla\times(\bm{\Pi}_h^{\bm E}\bm u-\bm u)),\nabla\times(\nabla\times e_h^{\bm u}))_{K}+(\nabla (\Pi_h^Qp-p), e_h^{\bm u})\nonumber\\
		&\quad\quad\quad-\sum_{F\in\mathcal{E}_h}\langle\{\!\!\{ \nabla\times(\nabla\times(\bm{\Pi}_h^{\bm E}\bm u-\bm u))\}\!\!\},\bm n\times[\![\nabla\times e_h^{\bm u}]\!] \rangle_F\nonumber\\
		&\quad\quad\quad-\sum_{F\in\mathcal{E}_h}\langle\{\!\!\{ \nabla\times(\nabla\times e_h^{\bm u})\}\!\!\},\bm n\times[\![\nabla\times(\bm{\Pi}_h^{\bm E}\bm u-\bm u)]\!] \rangle_F\nonumber\\
		&\quad\quad\quad+\sum_{F\in\mathcal{E}_h}\color{black}\tau h_F^{-1}\color{black}\langle\bm n\times[\![\nabla\times(\bm{\Pi}_h^{\bm E}\bm u-\bm u)]\!],\bm n\times[\![\nabla\times e_h^{\bm u}]\!] \rangle_F
		,\label{error31}\\
		&(e_h^{\bm u},\nabla e_h^p)=0.\label{error32}
		\end{align}
	\end{subequations}
	Utilizing \eqref{error3} and arguments similar to those in the proof of \Cref{stable0}, we arrive at
	\begin{align}\label{R0}
	&\frac{1}{2}\sum_{K\in\mathcal{T}_h}\|\nabla\times(\nabla\times e_h^{\bm u})\|^2_{0,K}
	+\frac{1}{2}\sum_{F\in\mathcal{E}_h}\color{black}\tau h_F^{-1}\color{black}\|\bm n\times[\![\nabla\times e_h^{\bm u}]\!]\|^2_{0,F}\nonumber\\
	&\qquad\le	\sum_{K\in\mathcal{T}_h}(\nabla\times(\nabla\times(\bm{\Pi}_h^{\bm E}\bm u-\bm u)),\nabla\times(\nabla\times e_h^{\bm u}))_{K}+(\nabla (\Pi_h^Qp-p), e_h^{\bm u})\nonumber\\
	&\qquad\quad-\sum_{F\in\mathcal{E}_h}\langle\{\!\!\{ \nabla\times(\nabla\times(\bm{\Pi}_h^{\bm E}\bm u-\bm u))\}\!\!\},\bm n\times[\![\nabla\times e_h^{\bm u}]\!] \rangle_F\nonumber\\
	&\qquad\quad-\sum_{F\in\mathcal{E}_h}\langle\{\!\!\{ \nabla\times(\nabla\times e_h^{\bm u})\}\!\!\},\bm n\times[\![\nabla\times(\bm{\Pi}_h^{\bm E}\bm u-\bm u)]\!] \rangle_F\nonumber\\
	&\qquad\quad+\sum_{F\in\mathcal{E}_h}\color{black}\tau h_F^{-1}\color{black}\langle\bm n\times[\![\nabla\times(\bm{\Pi}_h^{\bm E}\bm u-\bm u)]\!],\bm n\times[\![\nabla\times e_h^{\bm u}]\!] \rangle_F\nonumber\\
	&\qquad=:R_1+R_2+R_3+R_4+R_5.
	\end{align}
	Now, we make estimates $\{R_i\}_{i=1}^5$ individually. Let $\bm{\Pi}_{h,k}$ be the $L^2$-projection from $L^2(\Omega)$ to $[\mathcal{P}_k(\mathcal{T}_h)]^3$. Then,  by the triangle inequality, the inverse inequality and the estimate \eqref{app-u2},
	there holds
	\begin{align}\label{est-nabla}
	&\|\nabla\times(\nabla\times\bm{\Pi}_{h,k+1}^{\text{curl}}\bm u)-\nabla\times(\nabla\times\bm u)\|_{0,K}\nonumber\\
	&\quad\le
	\|\nabla\times(\nabla\times\bm{\Pi}_{h,k+1}^{\text{curl}}\bm u-\bm{\Pi}_{h,k}(\nabla\times\bm u))\|_{0,K}
	+
	\|\nabla\times(\bm{\Pi}_{h,k}(\nabla\times\bm u)-\nabla\times\bm u)\|_{0,K}\nonumber\\
	&\quad\le
	Ch_K^{-1}\|\nabla\times\bm{\Pi}_{h,k+1}^{\text{curl}}\bm u-\bm{\Pi}_{h,k}(\nabla\times\bm u)\|_{0,K}
	+
	|\bm{\Pi}_{h,k}(\nabla\times\bm u)-\nabla\times\bm u|_{1,K}\nonumber\\
	&\quad\le Ch_K^{s_{u_1}-1}\|\nabla\times\bm u\|_{s_{u_1},K}
	.\nonumber\\
	\end{align}
	From \eqref{def:piE} and the above estimate, we get
	\begin{align}\label{R1}
	|R_1| &=\left|\sum_{K\in\mathcal{T}_h}(\nabla\times(\nabla\times\bm{\Pi}_{h,k+1}^{\text{curl}}\bm u)-\nabla\times(\nabla\times\bm u),\nabla\times
	(\nabla\times e_h^{\bm u}   ))_K\right|\nonumber\\
	&\le Ch^{s_{u_1}-1}\|\nabla\times\bm u\|_{s_{u_1}}\left(\sum_{K\in\mathcal{T}_h}\|\nabla\times (\nabla\times e_h^{\bm u})\|_{0,K}\right)^{\frac{1}{2}}.
	\end{align}
	We use the approximation property for $\Pi_h^Q$ in \eqref{app-p} and \Cref{th:embed2} to obtain
	\begin{align}\label{R2}
	R_2&\le Ch^{s_p-1}\|p\|_{s_p}\|e_h^{\bm u}\|_0 \nonumber\\
	&\le Ch^{s_p-1}\|p\|_{s_p}\|e_h^{\bm u}\|_{0,3} \nonumber\\
	&\le Ch^{s_p-1}\|p\|_{s_p}\|\nabla\times e_h^{\bm u}\|_{0} \nonumber\\
	&\le Ch^{s_p-1}\|p\|_{s_p}\|\nabla\times e_h^{\bm u}\|_{0,3} \nonumber\\
	&\le Ch^{s_p-1}\|p\|_{s_p}
	\left(
	\sum_{K\in\mathcal{T}_h}\|\nabla\times(\nabla\times e_h^{\bm u})\|^2_{0,K}
	+\sum_{F\in\mathcal{E}_h}\color{black}\tau h_F^{-1}\color{black}\|\bm n\times[\![\nabla\times e_h^{\bm u}]\!]\|^2_{0,F}
	\right)^{\frac{1}{2}}
	.
	\end{align}
	Using the triangle inequality, the inverse inequality, and the approximation property of $\bm{\Pi}_{h,k+1}^{\text{curl}}$ in \eqref{app-u2} leads to
	\begin{align}
	&\sum_{F\in\mathcal{E}_h}h_F\|\{\!\!\{ \nabla\times(\nabla\times(\bm{\Pi}_{h,k+1}^{\text{curl}}\bm u-\bm u))\}\!\!\}\|_{0,F}^2\nonumber\\
	&\quad\le 2\sum_{F\in\mathcal{E}_h}h_F\|\{\!\!\{ \nabla\times(\nabla\times\bm{\Pi}_{h,k+1}^{\text{curl}}\bm u)-\bm{\Pi}_{h,k}\nabla\times(\nabla\times\bm u)\}\!\!\}\|_{0,F}^2
	\nonumber\\
	&\qquad+2\sum_{F\in\mathcal{E}_h}h_F\|\{\!\!\{ \bm{\Pi}_{h,k}\nabla\times(\nabla\times\bm u)\}-\nabla\times(\nabla\times\bm u)\}\!\!\}\|_{0,F}^2 \nonumber\\
	&\quad\le C\sum_{K\in\mathcal{T}_h}\|\ \nabla\times(\nabla\times\bm{\Pi}_{h,k+1}^{\text{curl}}\bm u)-\bm{\Pi}_{h,k}\nabla\times(\nabla\times\bm u)\|_{0,K}^2
	\nonumber\\
	&\qquad+Ch^{2s_{u_2}}\|\nabla\times(\nabla\times\bm u)\|_{s_{u_2}}^2 \nonumber\\
	&\quad\le C\sum_{K\in\mathcal{T}_h}\|\ \nabla\times(\nabla\times\bm{\Pi}_{h,k+1}^{\text{curl}}\bm u)-\nabla\times(\nabla\times\bm u)\|_{0,K}^2\nonumber\\
	&\qquad
	+C\sum_{K\in\mathcal{T}_h}\|\ \nabla\times(\nabla\times\bm u)-\bm{\Pi}_{h,k}\nabla\times(\nabla\times\bm u)\|_{0,K}^2
	+Ch^{2s_{u_2}}\|\nabla\times(\nabla\times\bm u)\|_{s_{u_2}}^2 \nonumber\\
	&\quad\le C\left(h^{s_{u_1}-1}\|\nabla\times\bm u\|_{s_{u_1}}+h^{s_{u_2}}\|\nabla\times(\nabla\times\bm u)\|_{s_{u_2}}\right)^2.\nonumber\\
	\end{align}
	We use the above estimate to get
	\begin{align}\label{R3}
	|R_3|&=\left| \sum_{F\in\mathcal{E}_h}\langle\{\!\!\{ \nabla\times\nabla\times(\bm{\Pi}_{h,k+1}^{\text{curl}}\bm u-\bm u)\}\!\!\},\bm n\times[\![\nabla\times e_h^{\bm u}]\!] \rangle_F\right| \nonumber\\
	&\le C\left(h^{s_{u_1}-1}\|\nabla\times\bm u\|_{s_{u_1}}
	+h^{s_{u_2}}\|\nabla\times\nabla\times\bm u\|_{s_{u_2}}
	\right)\nonumber\\
	&\quad\times
	\left(
	\sum_{F\in\mathcal{E}_h}\color{black}\tau h_F^{-1}\color{black}\|\bm n\times[\![\nabla\times e_h^{\bm u}]\!]\|^2_{0,F}
	\right)^{\frac{1}{2}}.
	\end{align}
	Again, by the triangle inequality, the approximation property of $\bm{\Pi}_{h,k+1}^{\text{curl}}$ in \eqref{app-u2}, we have
	\begin{align*}
	&\sum_{F\in\mathcal{E}_h}h_F^{-1}\|\bm n\times[\![\nabla\times(\bm{\Pi}_{h,k+1}^{\text{curl}}\bm u-\bm u)]\!]\|_{0,F}^2\nonumber\\
	&\quad\le
	2\sum_{F\in\mathcal{E}_h}h_F^{-1}\|\bm n\times[\![\nabla\times\bm{\Pi}_{h,k+1}^{\text{curl}}\bm u-\bm{\Pi}_{h,k}\nabla\times\bm u]\!]\|_{0,F}^2\nonumber\\
	&\qquad
	+2\sum_{F\in\mathcal{E}_h}h_F^{-1}\|\bm n\times[\![\bm{\Pi}_{h,k}\nabla\times\bm u-\nabla\times\bm u]\!]\|_{0,F}^2\nonumber\\
	&\quad\le C	\sum_{K\in\mathcal{T}_h}h_K^{-2}\|\nabla\times\bm{\Pi}_{h,k+1}^{\text{curl}}\bm u-\bm{\Pi}_{h,k}\nabla\times\bm u\|_{0,K}^2+Ch^{2(s_{u_1}-1)} \|\nabla\times\bm u\|_{s_{u_1}}^2\nonumber\\
	&\quad\le C	\sum_{K\in\mathcal{T}_h}h_K^{-2}\|\nabla\times\bm{\Pi}_{h,k+1}^{\text{curl}}\bm u-\nabla\times\bm u\|_{0,K}^2\nonumber\\
	&\qquad
	+C	\sum_{K\in\mathcal{T}_h}h_K^{-2}\|\nabla\times\bm u-\bm{\Pi}_{h,k}\nabla\times\bm u\|_{0,K}^2+Ch^{2(s_{u_1}-1)} \|\nabla\times\bm u\|_{s_{u_1}}^2\nonumber\\
	&\quad\le Ch^{2(\color{black}s_{u_1}-1)}\|\nabla\times\bm u\|_{s_{u_1}}^2.
	\end{align*}
	Using the above estimate and the inverse inequality, we arrive at
	\begin{align}\label{R4}
	|R_4|&=\left|	\sum_{F\in\mathcal{E}_h}\langle\{\!\!\{ \nabla\times(\nabla\times e_h^{\bm u})\}\!\!\},\bm n\times[\![\nabla\times(\bm{\Pi}_{h,k+1}^{\text{curl}}\bm u-\bm u)]\!] \rangle_F\right| \nonumber\\
	&\le Ch^{s_{u_1}-1}\|\nabla\times\bm u\|_{s_{u_1}}\left(\sum_{K\in\mathcal{T}_h}\|\nabla\times (\nabla\times e_h^{\bm u})\|_0\right)^{\frac{1}{2}},\\
	|R_5|&=\left|\sum_{F\in\mathcal{E}_h}\color{black}\tau h_F^{-1}\color{black}\langle\bm n\times[\![\nabla\times(\bm{\Pi}_{h,k+1}^{\rm curl}\bm u-\bm u)]\!],\bm n\times[\![\nabla\times e_h^{\bm u}]\!] \rangle_F\right| \nonumber\\
	&\le Ch^{s_{u_1}-1}\|\nabla\times\bm u\|_{s_{u_1}}\left(
	\sum_{F\in\mathcal{E}_h}\tau h_F^{-1}\|\bm n\times[\![\nabla\times e_h^{\bm u}]\!]\|^2_{0,F}
	\right)^{\frac{1}{2}}.
	\end{align}
From \eqref{R0}, \eqref{R1}, \eqref{R2}, \eqref{R3} and \eqref{R4}, it gives us that
	\begin{align*}
	&\sum_{K\in\mathcal{T}_h}\|\nabla\times(\nabla\times e_h^{\bm u})\|^2_{0,K}
	+\sum_{F\in\mathcal{E}_h}\color{black}\tau h_F^{-1}\color{black}\|\bm n\times[\![\nabla\times e_h^{\bm u}]\!]\|^2_{0,F}\nonumber\\
	&\qquad\le C
	(
	h^{s_{u_1}-1}\|\nabla\times\bm u\|_{s_{u_1}}
	+h^{s_{u_2}}\|\nabla\times(\nabla\times\bm u)\|_{s_{u_2}}
	+	h^{s_p-1}\|p\|_{s_p}
	).
	\end{align*}
The result for the estimates of $\bm u-\bm u_h$ can be concluded as follows
	\begin{align*}
	&\|e_h^{\bm u}\|_0+\|\nabla\times e_h^{\bm u}\|_0\nonumber\\
	&\qquad\le C\left(
	\sum_{K\in\mathcal{T}_h}\|\nabla\times(\nabla\times e_h^{\bm u})\|^2_{0,K}
	+\sum_{F\in\mathcal{E}_h}\color{black}\tau h_F^{-1}\color{black}\|\bm n\times[\![\nabla\times e_h^{\bm u}]\!]\|^2_{0,F}
	\right)^{\frac{1}{2}},
	\end{align*}
	according to \Cref{th:embed2} and the triangle inequality.
	Using \eqref{error11} with $\bm v_h=-\nabla e_h^p\in \bm E_h$ yields
	\begin{align*}
\|\nabla e_h^p\|_0^2
	&=	-(\nabla (p-\Pi_h^Qp),\nabla e_h^p)\le Ch^{s_p-1}\|p\|_{s_p}\|\nabla e_h^p\|_0,
	\end{align*}
	which  further provides us with completeness of the proof of the estimate for $\nabla (p-p_h)$ by using the triangle inequality.
	
\end{proof}

\section{$H({\rm curl})$ error estimate\label{sec:dual}}

To derive the $\bm H({\rm curl})$ error estimate, we need the following dual problem:
\par
Find $(\bm\Phi,\Psi)\in \bm H_0({\rm curl}^2;\Omega)\times H_0^1(\Omega)$ such that
\begin{subequations}\label{PDE:dual}
	\begin{align}
	\nabla\times(\nabla\times(\nabla\times(\nabla\times\bm \Phi)))+\nabla \Psi&=\bm \Theta&\text{ in }\Omega,\\
	\nabla\cdot \bm\Phi&=0&\text{ in }\Omega,\\
	\bm n\times \bm\Phi&=\bm 0 &\text{ on }\partial\Omega,\\
	\bm n\times(\nabla\times \bm\Phi)&=\bm 0&\text{ on }\partial\Omega,\\
	\Psi&=0&\text{ on }\partial\Omega.
	\end{align}
\end{subequations}
We notice that $\Psi=0$ when $\nabla\cdot\bm{\Theta}=0$.
\begin{assumption}\label{ass2} We assume that the following regularity \color{black} holds for the weak solution $(\bm{\Phi},\Psi)\in \bm H_0({\rm curl}^2;\Omega)\times H_0^1(\Omega)$ of \eqref{PDE:dual} \color{black}
	\begin{align}\label{reg}
	\|\bm{\Phi}\|_{\beta}
	+\|\nabla\times(\nabla\times\bm{\Phi})\|_{\beta}
	\color{black}
	+\|\nabla\times\bm\Phi\|_{1+\gamma}
	\le C_{\rm reg}\|\bm \Theta\|_0,
	\end{align}
where $\beta\in (\frac{1}{2},\color{black}1]$, $\gamma\in [0,1]$, $\gamma\le \beta$, and $C_{\rm reg}$ is a constant independent of mesh size.
\end{assumption}
We notice that when $\Omega$ is convex, \eqref{reg} holds as $ \gamma=1, \beta=1$ from the regularity result in \cite[Theorem 11]{MR3808156}.

\begin{lemma}  \label{lemma:dual}	\color{black} Let $(\bm u,p)\in \bm H_0({\rm curl}^2;\Omega)\times H_0^1(\Omega)$ be the weak solution of \eqref{PDE:orignial}, and let \Cref{ass}, \Cref{ass2} hold true; let $(\bm u_h, p_h)\in \bm E_h\times Q_h$ be the solution of \eqref{FEM}, \color{black} and let $\nabla\cdot\bm \Theta=0$. Then we have the following error estimates
	\begin{align*}
	(\bm \Theta,\bm u-\bm u_h)
	\le  Ch^{{\sigma}}
	(
	h^{s_{u_0}}\|\bm u\|_{s_{u_0}}
	+	h^{s_{u_1}-1}\|\nabla\times\bm u\|_{s_{u_1}}
	+   h^{s_{u_2}}\|\nabla\times(\nabla\times\bm u)\|_{s_{u_2}}
	+	h^{s_p-1}\|p\|_{s_p}
	)
	\|\bm{\Theta}\|_0,
	\end{align*}
	where $\sigma=\min(\beta,\gamma)$
	with $\beta$, $\gamma$ being defined in \eqref{reg}.	
\end{lemma}

\begin{proof} By arguments similar to those in the
	proof of \Cref{pipi}, we have the following equations \color{black} hold  for the weak solution $(\bm{\Phi},\Psi)\in \bm H_0({\rm curl}^2;\Omega)\times H_0^1(\Omega)$ of \eqref{PDE:dual} \color{black}:
	\begin{subequations}\label{equa:dual}
		\begin{align}
		\sum_{K\in\mathcal{T}_h}(\nabla\times(\nabla\times \bm\Phi),\nabla\times(\nabla\times(\bm u-\bm u_h)))_{K}+(\nabla \Psi,(\bm u-\bm u_h))&\nonumber\\
		-\sum_{F\in\mathcal{E}_h}\langle\{\!\!\{ \nabla\times(\nabla\times\bm \Phi)\}\!\!\},\bm n\times[\![\nabla\times(\bm u-\bm u_h)]\!] \rangle_F&\nonumber\\
		-\sum_{F\in\mathcal{E}_h}\langle\{\!\!\{ \nabla\times(\nabla\times(\bm u-\bm u_h))\}\!\!\},\bm n\times[\![\nabla\times\bm \Phi]\!] \rangle_F&\nonumber\\
		+\sum_{F\in\mathcal{E}_h}\color{black}\tau h_F^{-1}\color{black}\langle\bm n\times[\![\nabla\times\bm \Phi]\!],\bm n\times[\![\nabla\times(\bm u-\bm u_h)]\!] \rangle_F
		&=(\bm \Theta,(\bm u-\bm u_h)),\\
		(\bm \Phi,\nabla q)&=0.
		\end{align}
	\end{subequations}
	We use the fact $\Psi=0$ \color{black} to get \color{black}
	\begin{align}\label{proof:dual1}
	(\bm \Theta,\bm u-\bm u_h)
	&=\sum_{K\in\mathcal{T}_h}(\nabla\times(\nabla\times(\bm u-\bm u_h)),\nabla\times(\nabla\times \bm\Phi))_{K}\nonumber\\
	&\quad-\sum_{F\in\mathcal{E}_h}\langle\{\!\!\{ \nabla\times(\nabla\times\bm \Phi)\}\!\!\},\bm n\times[\![\nabla\times(\bm u-\bm u_h)]\!] \rangle_F\nonumber\\
	&\quad-\sum_{F\in\mathcal{E}_h}\langle\{\!\!\{ \nabla\times(\nabla\times(\bm u-\bm u_h))\}\!\!\},\bm n\times[\![\nabla\times\bm \Phi]\!] \rangle_F\nonumber\\
	&\quad+\sum_{F\in\mathcal{E}_h}\color{black}\tau h_F^{-1}\color{black}\langle\bm n\times[\![\nabla\times\bm \Phi]\!],\bm n\times[\![\nabla\times(\bm u-\bm u_h)]\!] \rangle_F.
	\end{align}

	Let $(\bm{\Phi}_h,\Psi_h)\in \bm E_h\times Q_h$ be the solution of the following system:  for $\forall\, (\bm v_h,q_h)\in \bm E_h\times Q_h$
	\begin{subequations}\label{FEM2}
		\begin{align}
		\sum_{K\in\mathcal{T}_h}(\nabla\times(\nabla\times\bm \Phi_h),\nabla\times(\nabla\times\bm v_h))_{K}+(\nabla \Psi_h,\bm v_h)&\nonumber\\
		-\sum_{F\in\mathcal{E}_h}\langle\{\!\!\{ \nabla\times(\nabla\times\bm \Phi_h)\}\!\!\},\bm n\times[\![\nabla\times\bm v_h]\!] \rangle_F&\nonumber\\
		-\sum_{F\in\mathcal{E}_h}\langle\{\!\!\{ \nabla\times(\nabla\times\bm v_h)\}\!\!\},\bm n\times[\![\nabla\times\bm \Phi_h]\!] \rangle_F&\nonumber\\
		+\sum_{F\in\mathcal{E}_h}\color{black}\tau h_F^{-1}\color{black}\langle\bm n\times[\![\nabla\times\bm \Phi_h]\!],\bm n\times[\![\nabla\times\bm v_h]\!] \rangle_F
		&=(\bm f,\bm v_h),\\
		(\bm \Phi_h,\nabla q_h)&=0
		\end{align}
	\end{subequations}
	Using \Cref{error estimates}, and the fact $\Psi=0$, we can get the error estimate:
	\begin{align}\label{error:dual}
	&\|\bm \Phi-\bm \Phi_h\|_0+\|\nabla\times(\bm \Phi-\bm \Phi_h)\|_0	+\|\nabla (\Psi-\Psi_h)\|_0\nonumber\\
	&\quad+\sum_{K\in\mathcal{T}_h}\|\nabla\times(\nabla\times (\bm \Phi-\bm \Phi_h))\|^2_{0,K}
	+\sum_{F\in\mathcal{E}_h}\color{black}\tau h_F^{-1}\color{black}\|\bm n\times[\![\nabla\times (\bm \Phi-\bm \Phi_h)]\!]\|^2_{0,F}
	\nonumber\\
	&\quad\quad	 \color{black}
	\le C\left(h^{\beta}
	\|\bm\Phi\|_{\beta}
	+h^{\beta}\|\nabla\times(\nabla\times\bm{\Phi})\|_{\beta}
	+h^{\gamma}\|\nabla\times\bm{\Phi}\|_{1+\gamma}\right).
	\end{align}
	From \eqref{error1}, it follows that
	\begin{subequations}\label{proof:dual2}
		\begin{align}
		\sum_{K\in\mathcal{T}_h}(\nabla\times(\nabla\times(\bm u-\bm u_h)),\nabla\times(\nabla\times\bm \Phi_h))_{K}+(\nabla (p-p_h),\bm \Phi_h)&\nonumber\\
		-\sum_{F\in\mathcal{E}_h}\langle\{\!\!\{ \nabla\times(\nabla\times(\bm u-\bm u_h))\}\!\!\},\bm n\times[\![\nabla\times\bm \Phi_h]\!] \rangle_F&\nonumber\\
		-\sum_{F\in\mathcal{E}_h}\langle\{\!\!\{ \nabla\times(\nabla\times\bm \Phi_h)\}\!\!\},\bm n\times[\![\nabla\times(\bm u-\bm u_h)]\!] \rangle_F&\nonumber\\
		+\sum_{F\in\mathcal{E}_h}\color{black}\tau h_F^{-1}\color{black}\langle\bm n\times[\![\nabla\times(\bm u-\bm u_h)]\!],\bm n\times[\![\nabla\times\bm \Phi_h]\!] \rangle_F
		&=0,\\
		(\bm u-\bm u_h,\nabla \Psi_h)&=0.
		\end{align}
	\end{subequations}
In terms of  \eqref{proof:dual1} and \eqref{proof:dual2}, and by the fact $(\nabla(p-p_h),\bm{\Phi}_h)=(\nabla(p-p_h),\bm{\Phi}_h-\bm{\Phi})$, we have
	\begin{align}\label{T0}
	(\bm \Theta,\bm u-\bm u_h)&=\sum_{K\in\mathcal{T}_h}(\nabla\times(\nabla\times(\bm u-\bm u_h)),\nabla\times(\nabla\times (\bm\Phi-\bm\Phi_h)))_{K}\nonumber\\
	&\quad-\sum_{F\in\mathcal{E}_h}\langle\{\!\!\{ \nabla\times(\nabla\times(\bm \Phi-\bm\Phi_h))\}\!\!\},\bm n\times[\![\nabla\times(\bm u-\bm u_h)]\!] \rangle_F\nonumber\\
	&\quad-\sum_{F\in\mathcal{E}_h}\langle\{\!\!\{ \nabla\times(\nabla\times(\bm u-\bm u_h))\}\!\!\},\bm n\times[\![\nabla\times(\bm \Phi-\bm\Phi_h)]\!] \rangle_F\nonumber\\
	&\quad+\sum_{F\in\mathcal{E}_h}\color{black}\tau h_F^{-1}\color{black}\langle\bm n\times[\![\nabla\times(\bm \Phi-\bm\Phi_h)]\!],\bm n\times[\![\nabla\times(\bm u-\bm u_h)]\!] \rangle_F
	\nonumber\\
	&\quad+(\nabla (p-p_h),\bm\Phi-\bm \Phi_h)
	\nonumber\\
	&=:T_1+T_2+T_3+T_4+T_5.
	\end{align}
	Now we make the estimate for $\{T_i\}_{i=1}^5$ individually. To simply the notation, we define
	\begin{align*}
	\mathcal M:=(
	h^{s_{u_0}}\|\bm u\|_{s_{u_0}}
	+	h^{s_{u_1}-1}\|\nabla\times\bm u\|_{s_{u_1}}
	\color{black}
		+	h^{s_{u_2}}\|\nabla\times(\nabla\times\bm u)\|_{s_{u_2}}
	+	h^{s_p-1}\|p\|_{s_p}
	).
	\end{align*}
	We use the Cauchy-Schwartz inequality,
	the estimate in \Cref{error estimates}, the error estimate \eqref{error:dual} and the regularity \eqref{reg}
	to get
	\begin{align}\label{T1}
	|T_1|&\le \left(
	\sum_{K\in\mathcal{T}_h}\|\nabla\times(\nabla\times(\bm u-\bm u_h)   )\|_{0,K}
	\right)^{\frac{1}{2}}
	\left(
	\sum_{K\in\mathcal{T}_h}\|\nabla\times(\nabla\times(\bm \Phi-\bm \Phi_h)   )\|_{0,K}
	\right)^{\frac{1}{2}} \nonumber\\
	&\le
	C\mathcal M \color{black}
	\left( h^{\beta}
	\|\bm\Phi\|_{\beta}
	+h^{\beta}\|\nabla\times(\nabla\times\bm{\Phi})\|_{\beta}
	+h^{\gamma}\|\nabla\times\bm{\Phi}\|_{1+\gamma}
	\right)\nonumber\\
	&\le  Ch^{{\sigma}}\mathcal{M}\|\bm{\Theta}\|_0.
	\end{align}
	Using  the Cauchy-Schwartz inequality, the triangle inequality, the inverse inequality, the error estimates in \eqref{error:dual}, \eqref{error estimates}, and the regularity \eqref{reg},
	we get
	\begin{align}\label{T2}
	|T_2|&\le \sum_{F\in \mathcal{E}_h}
	h_F^{\frac{1}{2}}\|	\{\!\!\{ \nabla\times(\nabla\times(\bm \Phi-\bm\Phi_h))\}\!\!\}\|_{0,F}
	h_F^{-\frac{1}{2}}\|\bm n\times[\![\nabla\times(\bm u-\bm u_h)]\!] \|_{0,F}\nonumber\\
	&\le \sum_{F\in \mathcal{E}_h}
	\|	\{\!\!\{ \nabla\times(\nabla\times\bm\Phi-\bm{\Pi}_{h,k+1}^{\rm curl}(\nabla\times\bm \Phi))\}\!\!\}\|_{0,F}
	\|\bm n\times[\![\nabla\times(\bm u-\bm u_h)]\!] \|_{0,F}\nonumber\\
	&\quad+\sum_{F\in \mathcal{E}_h}
	\|	\{\!\!\{ \nabla\times(\bm{\Pi}_{h,k+1}^{\rm curl}(\nabla\times\bm \Phi)-\nabla\times\bm\Phi_h))\}\!\!\}\|_{0,F}
	\|\bm n\times[\![\nabla\times(\bm u-\bm u_h)]\!] \|_{0,F}\nonumber\\
	&\le Ch^{\beta}\mathcal M\|\nabla\times(\nabla\times\bm\Phi)\|_{\beta}
	+
	Ch^{\gamma}\mathcal M\|\nabla\times\bm\Phi\|_{1+\gamma}
	\nonumber\\
	&\quad	+C\mathcal M\left(
	\sum_{K\in\mathcal{T}_h}
	\|\nabla\times(\bm{\Pi}_{h,k+1}^{\rm curl}(\nabla\times\bm \Phi)-\nabla\times\bm\Phi_h))\|_{0,K}
	\right)^{\frac{1}{2}}\nonumber\\
	&\le
	C\mathcal M
	\left( h^{\beta}
	\|\bm\Phi\|_{\beta}+h^{\beta}\|\nabla\times(\nabla\times\bm{\Phi})\|_{\beta}
	+h^{\gamma}\|\nabla\times\bm{\Phi}\|_{1+\gamma}
	\right)\nonumber\\
	&\le  Ch^{{\sigma}}\mathcal{M}\|\bm{\Theta}\|_0.
	\end{align}
	Similar to \eqref{T2}, one can get
	\begin{align}\label{T3}
	|T_3|&\le \sum_{F\in \mathcal{E}_h}
	\|	\{\!\!\{ \nabla\times(\nabla\times(\bm u-\bm u_h))\}\!\!\}\|_{0,F}
	\|\bm n\times[\![\nabla\times(\bm \Phi-\bm \Phi_h)]\!] \|_{0,F}\nonumber\\
	&\le \sum_{F\in \mathcal{E}_h}
	\|	\{\!\!\{ \nabla\times(\nabla\times\bm u-\bm{\Pi}_{h,k+1}^{\rm curl}(\nabla\times\bm u))\}\!\!\}\|_{0,F}
	\|\bm n\times[\![\nabla\times(\bm \Phi-\bm \Phi_h)]\!] \|_{0,F}\nonumber\\
	&\quad+\sum_{F\in \mathcal{E}_h}
	\|	\{\!\!\{ \nabla\times(\bm{\Pi}_{h,k+1}^{\rm curl}(\nabla\times\bm u)-\nabla\times\bm u_h))\}\!\!\}\|_{0,F}
	\|\bm n\times[\![\nabla\times(\bm \Phi-\bm \Phi_h)]\!] \|_{0,F}\nonumber\\
	&\le
	C\mathcal M \color{black}
	\left( h^{\beta}
	\|\bm\Phi\|_{\beta}
	+h^{\beta}\|\nabla\times(\nabla\times\bm{\Phi})\|_{\beta}
	+h^{\gamma}\|\nabla\times\bm{\Phi}\|_{1+\gamma}
	\right)\nonumber\\
	&\le  Ch^{{\sigma}}\mathcal{M}\|\bm{\Theta}\|_0,\\
\label{T4}
	|T_4|	&\le
	C\mathcal M \color{black}
	\left( h^{\beta}
	\|\bm\Phi\|_{\beta}
	+h^{\beta}\|\nabla\times(\nabla\times\bm{\Phi})\|_{\beta}
	+h^{\gamma}\|\nabla\times\bm{\Phi}\|_{1+\gamma}
	\right)\le  Ch^{{\sigma}}\mathcal{M}\|\bm{\Theta}\|_0.
	\end{align}
	Again, we use the Cauchy-Schwartz inequality, the estimate in \Cref{error estimates}, the error estimates in \eqref{error:dual} and the regularity \eqref{reg} to get
	\begin{align}\label{T5}
	|T_5|&\le \|\nabla p-\nabla p_h\|_0\|\bm{\Phi}-\bm\Phi_h\|_0\nonumber\\
	&\le
	C\mathcal M \color{black}
	\left( h^{\beta}\|\bm{\Phi}\|_{\beta}
	+h^{\beta}\|\nabla\times(\nabla\times\bm{\Phi})\|_{\beta}
	+h^{\gamma}\|\nabla\times\bm{\Phi}\|_{1+\gamma}
	\right)\nonumber\\
	&\le  Ch^{{\sigma}}\mathcal{M}\|\bm{\Theta}\|_0.
	\end{align}
	We thus complete the proof  by  applying \eqref{T0}, \eqref{T1}, \eqref{T2}, \eqref{T3}, \eqref{T4} and \eqref{T5}.
\end{proof}

In the following two subsections,  we will give estimates for
$\|\bm u-\bm u_h\|_0$ and $\|\nabla\times(\bm u-\bm u_h)\|_0$, respectively,  and then obtain the error estimate in $\bm H({\rm curl})$ norm.

\subsection{$L^2$ error estimate}

\begin{theorem}\label{dual-u} 	\color{black} Let $(\bm u,p)\in \bm H_0({\rm curl}^2;\Omega)\times H_0^1(\Omega)$ be the weak solution of \eqref{PDE:orignial}, and let \Cref{ass}, \Cref{ass2} hold true; let $(\bm u_h, p_h)\in \bm E_h\times Q_h$ be the solution of \eqref{FEM}. \color{black} Then we have the following error estimates
	\begin{align}
	\|\bm u-\bm u_h\|_0\le&
	Ch^{{\sigma}}
	(	h^{s_{u_1}-1}\|\nabla\times\bm u\|_{s_{u_1}}
	\color{black}
	+h^{s_{u_2}}\|\nabla\times(\nabla\times\bm u)\|_{s_{u_2}}
	+	h^{s_p-1}\|p\|_{s_p}
	) +Ch^{s_{u_0}}\|\bm u\|_{s_{u_0}},
	\end{align}	
	where $\sigma=\min(\alpha,\beta,\gamma)$,
	$\alpha$ is defined in \Cref{embed},
	$\beta$ and $\gamma$ are defined in \eqref{reg}.
	
\end{theorem}
\begin{proof}
	
	We take $\bm{\Theta}$ satisfying the following problem:
	\par
	Find $\bm \Theta\in \bm H({\rm curl};\Omega)\cap \bm H({\rm div};\Omega)$ such that
	\begin{align}\label{theta}
	\nabla\times\bm\Theta&=\nabla\times(\bm{\Pi}_h^{\bm E}\bm u-\bm u_h)&\text{in }\Omega,\\
	\nabla\cdot\bm \Theta&=0&\text{in }\Omega,\\
	\bm n\times\bm \Theta&=\bm 0&\text{on }\partial\Omega.
	\end{align}
	It follows form \eqref{fem01} and \eqref{or-p} that
	\begin{align*}
	(\bm{\Pi}_h^{\bm E}\bm u-\bm u_h,\nabla q_h)=(\bm u,\nabla q_h)=-(\nabla\cdot\bm u,q_h)=0
	\end{align*}
	for all $q_h\in Q_h$.
	Due to the result in \cite[Lemma 4.5]{MR2009375},  one has
	\begin{align}
	\|(\bm{\Pi}_h^{\bm E}\bm u-\bm u_h)-\bm{\Theta}\|_0&\le Ch^{\alpha}\|\nabla\times(\bm{\Pi}_h^{\bm E}\bm u-\bm u_h)\|_{0},\label{theta02}
	\end{align}
	where $\alpha$ is defined in \Cref{embed}. As a result, using the triangle inequality, the estimate \eqref{theta02}, and $\sigma\le \alpha$, one can get
	\begin{align}\label{est:theta01}
	\|\bm\Theta\|_0&\le \|(\bm{\Pi}_h^{\bm E}\bm u-\bm u_h)-\bm{\Theta}\|_0+
	\|\bm{\Pi}_h^{\bm E}\bm u-\bm u_h\|_0\nonumber\\
	&\le Ch^{\sigma}\|\nabla\times(\bm{\Pi}_h^{\bm E}\bm u-\bm u_h)\|_{0}+\|\bm u-\bm u_h\|_0 \nonumber\\
	&\quad +Ch^{s_{u_0}}(\|\bm u\|_{s_{u_0}}
	+\delta(s_{u_0})h\|\nabla\times\bm u\|_{s_{u_0}}
	)\nonumber\\
	&\le Ch^{\sigma}\mathcal M+Ch^{s_{u_0}}(\|\bm u\|_{s_{u_0}}
	+\delta(s_{u_0})h\|\nabla\times\bm u\|_{s_{u_0}}
	)+\|\bm u-\bm u_h\|_0.
	\end{align}
 Due to the triangle inequality, \eqref{est:theta01}, \Cref{lemma:dual} and \eqref{theta02}, one  arrives at
	\begin{align*}
	\|\bm u-\bm u_h\|_0^2&=(\bm u-\bm u_h,\bm u-\bm u_h)\nonumber\\
	&=(\bm \Theta,\bm u-\bm u_h)
	+((\bm{\Pi}_h^{\bm E}\bm u-\bm u_h)-\bm \Theta,\bm u-\bm u_h)
	+(\bm u-\bm{\Pi}_h^{\bm E}\bm u,\bm u-\bm u_h)\nonumber\\
	&\le Ch^{\sigma}\mathcal{M}\|\bm{\Theta}\|_0
	+Ch^{\sigma}\|\nabla\times(\bm{\Pi}_h^{\bm E}\bm u-\bm u_h)\|_{0}\|\bm u-\bm u_h\|_0
	\nonumber\\
	&\quad+Ch^{s_{u_0}}(\|\bm u\|_{s_{u_0}}
	+\delta(s_{u_0})h\|\nabla\times\bm u\|_{s_{u_0}}
	)
	\|\bm u-\bm u_h\|_0\nonumber\\
	&\le Ch^{2\sigma}\mathcal{M}^2
	+Ch^{2s_{u_0}}(\|\bm u\|_{s_{u_0}}
	+\delta(s_{u_0})h\|\nabla\times\bm u\|_{s_{u_0}}
	)^2+\frac{1}{2}\|\bm u-\bm u_h\|_0^2,
	\end{align*}
	where we have defined $\delta(s_{u_0})=0$ when $s_{u_0}>1$ and $\delta(s_{u_0})=1$ when $s_{u_0}\in (\frac{1}{2},1]$. The above inequality further implies that
	\begin{align*}
	\|\bm u-\bm u_h\|_0\le& C\left(h^{\sigma}
	(	h^{s_{u_1}-1}\|\nabla\times\bm u\|_{s_{u_1}}
	+	h^{s_p-1}\|p\|_{s_p}
	)+h^{s_{u_0}}(\|\bm u\|_{s_{u_0}}
	+\delta(s_{u_0})h\|\nabla\times\bm u\|_{s_{u_0}}
	)\right)\nonumber\\
	\le& C\left(h^{\sigma}
	(	h^{s_{u_1}-1}\|\nabla\times\bm u\|_{s_{u_1}}
	+	h^{s_p-1}\|p\|_{s_p}
	)+h^{s_{u_0}}\|\bm u\|_{s_{u_0}}\right),
	\end{align*}
 and we complete the proof.
\end{proof}

\subsection{Curl operator error estimate}

We first introduce an interpolation denoted by $\bm{\Pi}_{h,\ell}^{\rm curl,c}$. Form \cite[Proposition 4.5]{MR2194528}, for any integer $\ell\ge 1$, let $\bm v_h\in [\mathcal P_{\ell}(\mathcal{T}_h)]^3$, there exists a function $\bm{\Pi}_{h,\ell}^{\rm curl,c}\bm v_h\in
[\mathcal P_{\ell}(\mathcal{T}_h)]^3\cap \bm H_0({\rm curl};\Omega)
$ such that
\begin{align*}
\|\bm{\Pi}_{h,\ell}^{\rm curl,c}\bm v_h-\bm v_h\|_0
+h\left(\sum_{K\in\mathcal{T}_h}\|\nabla\times(\bm{\Pi}_{h,\ell}^{\rm curl,c}\bm v_h-\bm v_h)\|_{0,K}^2\right)^{\frac{1}{2}}
\le C
\left(
\sum_{F\in\mathcal{E}_h}h_F\|\bm n\times[\![
\bm v_h
]\!]\|_{0,F}^2
\right)^{\frac{1}{2}}.
\end{align*}

\begin{theorem}
	\color{black} Let $(\bm u,p)\in \bm H_0({\rm curl}^2;\Omega)\times H_0^1(\Omega)$ be the weak solution of \eqref{PDE:orignial}, and let \Cref{ass}, \Cref{ass2} hold true; let $(\bm u_h, p_h)\in \bm E_h\times Q_h$ be the solution of \eqref{FEM}. \color{black} Then we have the following error estimates
	\begin{align*}\small
	\|\nabla\times(\bm u-\bm u_h)\|_0\le&
	Ch^{\sigma}
	(	
		h^{s_{u_1}-1}\|\nabla\times\bm u\|_{s_{u_1}}
	\color{black}
		+	h^{s_{u_2}}\|\nabla\times(\nabla\times\bm u)\|_{s_{u_2}}
	+	h^{s_p-1}\|p\|_{s_p}
	)
	\\
	&\color{black}+
	Ch^{s_{u_0}}\|\bm u\|_{s_{u_0}},
	\end{align*}	
	where $\sigma=\min(\alpha,\beta,\gamma)$,
	$\alpha$ is defined in \Cref{embed},
	$\beta$ and $\gamma$ are defined in \eqref{reg}.
	
\end{theorem}
The proof of the theorem  can be obtained through arguments similar to those in the proof of \Cref{dual-u} and through considering the problem \eqref{theta} with $\nabla\times(\bm{\Pi}_{h,k+1}^{\rm curl}(\nabla\times \bm u)-\bm{\Pi}_{h,k+1}^{\rm curl,c}(\nabla\times\bm u_h))$ being the right hand side term, and we skip the details of the proof. Finally, we present the optimal error estimates in $\bm H({\rm curl})$-norm as follows.

\begin{proposition} \label{pro}
	\color{black} Let $(\bm u,p)\in \bm H_0({\rm curl}^2;\Omega)\times H_0^1(\Omega)$ be the weak solution of \eqref{PDE:orignial}, and $(\bm u_h, p_h)\in \bm E_h\times Q_h$ be the solution of \eqref{FEM}. \color{black}
	When $\Omega$ is convex and \color{black} the weak solution $(\bm u,p)\in \bm H_0({\rm curl}^2;\Omega)\times H_0^1(\Omega)$ of \eqref{PDE:orignial} is \color{black} sufficiently smooth,
	then we have the following error estimates
	\begin{align*}
	\|\bm u-\bm u_h\|_0+\|\nabla\times(\bm u-\bm u_h)\|_0\le Ch^{k+1}
	(\|\bm u\|_{k+1} +\|\nabla\times\bm u\|_{k+1}+\|p\|_{k+1}   ).
	\end{align*}
\end{proposition}

{
\color{black}
\section{Numerical experiments}
We consider the nonhomogeneous problem:
\begin{subequations}\label{PDE:orignial-non}
	\begin{align}
	\nabla\times(\nabla\times(\nabla\times(\nabla\times\bm u)))+\nabla p&=\bm f&\text{ in }\Omega,\\
	\nabla\cdot\bm u&=0&\text{ in }\Omega,\\
	\bm n\times\bm u&=\bm g_T &\text{ on }\partial\Omega,\\
	\bm n\times(\nabla\times\bm u)&=\bm m_T&\text{ on }\partial\Omega,\\
	p&=0&\text{ on }\partial\Omega,
	\end{align}
\end{subequations}
instead of \eqref{PDE:orignial}. As in \cite{MR3771897}, we take $\Omega=[0,1]^3$ and the exact solution $\bm u=(u_1,u_2,u_3)^T$ and $p$ read
\begin{align*}
u_1&=\sin (\pi y)\sin (\pi z),\\
u_2&=\sin (\pi z)\sin (\pi x),\\
u_3&=\sin (\pi x)\sin (\pi y),\\
p&=\sin(2\pi x)\sin(2\pi y)\sin(2\pi z).
\end{align*}
The functions
$\bm{f}$, $\bm g_T$ and $\bm m_T$ are determined according to \eqref{PDE:orignial-non} and the above exact solutions. We modify the scheme \eqref{FEM} to approximate the nonhomogeneous problem \eqref{PDE:orignial-non} as follows:

Find $\bm u_h\in \bm E_h^{\bm g_T}$ and $p_h\in Q_h$ such that
\begin{subequations}\label{FEM-non}
	\begin{align}
	\sum_{K\in\mathcal{T}_h}(\nabla\times(\nabla\times\bm u_h),\nabla\times(\nabla\times\bm v_h))_{K}+(\nabla p_h,\bm v_h)&\nonumber\\
	-\sum_{F\in\mathcal{E}_h}\langle\{\!\!\{ \nabla\times(\nabla\times\bm u_h)\}\!\!\},\bm n\times[\![\nabla\times\bm v_h]\!] \rangle_F&\nonumber\\
	\pm\sum_{F\in\mathcal{E}_h}\langle\{\!\!\{ \nabla\times(\nabla\times\bm v_h)\}\!\!\},\bm n\times[\![\nabla\times\bm u_h]\!] \rangle_F&\nonumber\\
	+\sum_{F\in\mathcal{E}_h}\tau h_F^{-1}\langle\bm n\times[\![\nabla\times\bm u_h]\!],\bm n\times[\![\nabla\times\bm v_h]\!] \rangle_F&\nonumber\\
	=(\bm f,\bm v_h)\pm\sum_{F\in\mathcal{E}_h\cap\partial\Omega}\langle\bm m_T, \nabla\times(\nabla\times\bm v_h)\rangle_{F}	&\nonumber\\
	+\sum_{F\in\mathcal{E}_h\cap\partial\Omega}\tau h_F^{-1}\langle \bm m_T,\bm n\times\nabla\times\bm v_h \rangle_F&,\\
	(\bm u_h,\nabla q_h)&=0
	\end{align}
	hold for all $(\bm v_h,q_h)\in \bm E_h\times Q_h$, where $\bm E_h^{\bm g_T}$ is defined as
	\begin{align*}
	\bm E_h^{\bm g_T}:=\{\bm v\in  [\mathcal P_{k+1}(\mathcal{T}_h)]^3:\bm n\times\bm v|_{\partial\Omega}=\bm{\Pi}_h^{\rm div}\bm g_T\},
	\end{align*}
	and $\bm{\Pi}_h^{\rm div}$ is some $H(\rm div)$-projection onto $k^{th}$ order piecewise polynomial spaces defined on $\mathcal E_h\cap\partial\Omega$.
\end{subequations}

\begin{table}[h]
	\caption{History of convergence for $k=0,1,2$\label{tab}}
	\centering
	\begin{tabular}{c|c|c|c|c|c|c}	
		\Xhline{1pt}
		\multirow{2}{*}{$k$} &
		\multirow{2}{*}{$h^{-1}$} &
		\multicolumn{2}{c|}{$\|\bm u-\bm u_h\|_0$} &
		\multicolumn{2}{c|}{$\|p-p_h\|_0$} &
		\multirow{2}{*}{DOF} \\
		\cline{3-6}
		
		& &Error &Rate  &Error &Rate &  \\
		\Xhline{1pt}

&	2	&	5.69E-01	&		&	2.60E-01	&		&	321 	\\
&	4	&	4.18E-01	&	0.44 	&	5.05E-02	&	2.36 	&	1937 	\\
0&	8	&	3.27E-01	&	0.35 	&	6.25E-03	&	3.01 	&	13281 	\\
&	16	&	2.81E-01	&	0.22 	&	7.37E-04	&	3.08 	&	97985 	\\
&	24	&	2.66E-01	&	0.14 	&	2.15E-04	&	3.04 	&	321697 	\\

		\Xhline{1pt}
		
&	2	&	1.48E-01	&	    	&	9.72E-02	&		&	997 	\\
&	4	&	5.72E-02	&	1.37 	&	8.77E-03	&	3.47 	&	6601 	\\
1&	8	&	1.96E-02	&	1.54 	&	5.57E-04	&	3.98 	&	47761 	\\
&	10	&	1.34E-02	&	1.72 	&	2.23E-04	&	4.09 	&	91381 	\\
&	12	&	9.65E-03	&	1.79 	&	1.06E-04	&	4.09 	&	155737 	\\

		\Xhline{1pt}
&	2	&	2.70E-02	&		&	3.18E-02	&		&	2273 	\\
&	4	&	2.61E-03	&	3.37 	&	1.44E-03	&	4.46 	&	15777 	\\
2&	6	&	5.72E-04	&	3.51 	&	1.99E-04	&	4.62 	&	50689 	\\
&	8	&	1.88E-04	&	3.79 	&	4.89E-05	&	4.88 	&	117185 	\\

		\Xhline{1pt}
	\end{tabular}	
\end{table}
We simply compute the symmetry case  and take $\tau=10$ for the numerical implementation.
The convergence results for $k=0,1,2$ are presented in \Cref{tab} (although we have not performed analysis for the case of $k=0$, we present the numerical results for a purpose of complete illustration). We can observe that the convergence rate  for $p_h$  is $k+3$  which is optimal;
the convergence rate for $\bm u_h$ is $0$ when $k=0$;
the convergence rate for $\bm u_h$ is $2$  which is not optimal when $k=1$ ;
 the convergence rate for $\bm u_h$ is $4$  which is optimal when $k=2$ .
All these results are in a good agreement with the theoretical analysis results.
}
\section*{\color{black}Acknowledgments}
{\small
Gang Chen is supported by National Natural Science Foundation of China (NSFC) under grant no. 11801063,
China Postdoctoral Science Foundation under grant no. 2018M633339 {\color{black} and 2019T120828}.
Weifeng Qiu is supported by Research Grants Council of the Hong Kong Special Administrative Region of China under grant no. CityU 11304017.
Liwei Xu is supported by a Key Project of the Major Research Plan of NSFC under grant no. 91630205 and NSFC under grant no. 11771068.
}

\bibliographystyle{siam}

\end{document}